\documentclass[reqno,10pt]{amsart}
\usepackage{amsmath, amssymb, amsthm, mathtools}
\usepackage[margin=1in]{geometry}
\usepackage{url}
\usepackage[dvipsnames]{xcolor}
\usepackage{ytableau}

\numberwithin{equation}{section}
{\theoremstyle{plain}
\newtheorem{theorem}{Theorem}[section]
\newtheorem{lemma}[theorem]{Lemma}

\newtheorem{proposition}[theorem]{Proposition}
\newtheorem{example}[theorem]{Example}
}
{\theoremstyle{definition}
\newtheorem{definition}[theorem]{Definition}
}
{\theoremstyle{remark}
\newtheorem{remark}[theorem]{Remark}
}
\usepackage[hidelinks]{hyperref}
\hypersetup{
    colorlinks=true,
    linkcolor=blue,
    urlcolor=red,
    citecolor=Green
    }
\usepackage[backend=biber, style=alphabetic, maxbibnames=999, maxalphanames=999,doi=true,
  url=false]{biblatex}
\DeclareLabelalphaTemplate{%
  \labelelement{\field[final]{shorthand}}%
  \labelelement{\field[uppercase,strwidth=1,strside=left]{labelname}}%
  \labelelement{\field[strwidth=2,strside=right]{year}\field{extrayear}}%
}
\AtBeginDocument{%
}

\newcommand{\deftobe}{\coloneq} 

\newcommand{\Prob}{\mathbb{P}}
\newcommand{\Inv}{\operatorname{Inv}}
\newcommand{\hookprod}{\pi}
\newcommand{\e}{\mathbb{E}}
\newcommand{\beq}{\preceq} 
\newcommand{\sbo}{\preceq_B}
\newcommand{\wbo}{\preceq_W} 
\newcommand{\strongintervals}[1]{\mathcal{B}_{#1}} 
\newcommand{\weakintervals}[1]{\mathcal{W}_{#1}} 
\newcommand{\sym}[1]{\mathfrak{S}_{#1}} 
\newcommand{\GL}{\mathrm{GL}} 
\newcommand{\gale}{\leq_G}
\newcommand{\RSK}{\mathrm{RSK}} 
\newcommand{\SYT}{\mathrm{SYT}} 
\newcommand{\planch}{\Prob_{\mathrm{Pl}}} 
\newcommand{\conj}[1]{#1'} 
\newcommand{\hook}[3]{h_{#1}(#2,#3)} 
\newcommand{\hookproduct}[1]{\pi(#1)} 
\newcommand{\permposet}[1]{P(#1)}

\definecolor{forgreen}{RGB}{30, 150, 30}

\title{Comparability in Bruhat orders}

\author[Boretsky]{Jonathan Boretsky}
\address{Department of Mathematics and Statistics, McGill University, Montreal, QC}
\email{\href{mailto:jonathan.boretsky@mcgill.ca}{jonathan.boretsky@mcgill.ca}}

\author[Cornejo]{Alvaro Cornejo}
\address{Department of Mathematics, University of Kentucky, Lexington, KY 40508}
\email{\href{mailto:alvaro.cornejo@uky.edu}{alvaro.cornejo@uky.edu}}

\author[Hodges]{Reuven Hodges}
\address{Department of Mathematics, University of Kansas, Lawrence, KS 66045}
\email{\href{mailto:rmhodges@ku.edu}{rmhodges@ku.edu}}

\author[Horn]{Paul Horn}
\address{Department of Mathematics, University of Denver, Denver, CO 80210 and Department of Mathematics and Applied Mathematics, University of Johannesburg, Johannesburg, South Africa}
\email{\href{mailto:paul.horn@du.edu}{paul.horn@du.edu}}

\author[Lesnevich]{Nathan Lesnevich}
\address{Department of Mathematics, Oklahoma State University, Stillwater, OK 74078}
\email{\href{mailto:nlesnev@okstate.edu}{nlesnev@okstate.edu}}

\author[McAllister]{Tyrrell McAllister}
\address{Department of Mathematics and Statistics, University of Wyoming, Laramie, WY, 82071}
\email{\href{mailto:tmcallis@uwyo.edu}{tmcallis@uwyo.edu}}

\begin{document}

\begin{abstract}
    We determine the sharp asymptotic scale of the probability that two uniformly random permutations are comparable in weak Bruhat order, showing that 
    $\mathbb{P}(\sigma_1 \preceq_W \sigma_2)=\exp\Bigl(\bigl(-\tfrac12+o(1)\bigr)\,n\log n\Bigr)$. This significantly improves both of the best known bounds, due to Hammett and Pittel, which placed this probability between $\exp((-1+o(1))n\log n)$ and $\exp(-\Theta(n))$. We also improve the best known lower bound for strong Bruhat-order comparability, due to the same authors, by proving a subexponential lower bound. The Bruhat orders are natural partial orders on the symmetric group, appearing in wide-reaching settings including the geometry of flag manifolds, the representation theory of $\mathfrak{S}_{n}$, and the combinatorics of the permutohedron. To analyze weak Bruhat order, we combine classic analytic, tableau-theoretic, and poset-theoretic tools, including the Plancherel measure and the RSK bijection. For strong Bruhat order we construct large families where members are comparable with high probability.  Our proof that members are comparable combines the tableau criterion with an associated random-walk-type deviation process. 
\end{abstract}

\maketitle

\section{Introduction} 

When a family of objects carries a natural way to permute its underlying labels, that is, an action
of the symmetric group $\sym{n}$, the action leaves a shadow of $\sym{n}$: the group’s internal structure
imposes systematic constraints on the objects’ underlying combinatorics. This shadow appears in many
forms. In algebra it governs the invariant theory and representation theory of $\sym{n}$ naturally attached to $\sym{n}$-actions;
in geometry and topology it enters through configuration spaces, flag varieties, and Schubert Calculus,
where permutations label Schubert strata and their closure relations; and in combinatorics,
probability, and algorithms it sheds light on statistics of random permutations and sorting by adjacent
transpositions \cite{BB05,Humphreys90}. Across these settings, understanding the $\sym{n}$-action naturally leads one to compare permutations in a manner compatible with the standard generators of $\sym{n}$, namely the adjacent transpositions. 
The strong and weak Bruhat orders furnish two canonical comparability relations on $\sym{n}$. These intrinsic partial orders are tailored to different ways of assembling permutations from these generators, encoding geometric closure phenomena and one-sided monotone dynamics, respectively.

The strong Bruhat order, denoted here by $\sbo$ and defined in Section~\ref{sec:background}, traces its origins to the decades-long program—catalyzed by Hilbert’s fifteenth
problem~\cite{H02}—to place Schubert’s enumerative calculus on a rigorous, axiomatic footing. Permutations index Schubert cells in the type-A flag manifold $\GL_n/B$, and already in Ehresmann’s 1934 work one sees that
the closure relations among these cells induce a partial order on $\sym{n}$ \cite{Ehresmann34}.
Chevalley’s formulation of the Bruhat decomposition 
\[
G=\bigsqcup_{w\in \mathfrak{W}}BwB
\]
for any semisimple group $G$ with Weyl group $\mathfrak{W}$ makes this order intrinsic: $u\sbo v$ records containment of Schubert varieties
\cite{C94}. From Verma’s representation-theoretic study \cite{Verma71} onward, Bruhat order became a
basic tool, admitting combinatorial characterizations through reduced words and reflections; in particular,
the subword criterion relates comparability to containment of reduced expressions \cite{BB05}.
Later work in algebraic combinatorics clarified the topology of intervals (\emph{e.g.}, EL-shellability) and
strengthened the links to Hecke algebras and Kazhdan--Lusztig theory \cite{BjornerWachs82,KazhdanLusztig79}. Many of these algebraic and geometric connections extend, in surprising generality, to all Weyl, reflection, and even Coxeter groups.

The (right) weak Bruhat order, denoted here by $\wbo$ and defined in Section~\ref{sec:background}, is a coarsening of the strong Bruhat order; in particular, $\sigma \wbo \omega$ implies $\sigma \sbo \omega$. The weak Bruhat order admits left and right variants, which induce isomorphic posets on $\sym{n}$; throughout we work with the right weak order. The relationship between the strong and weak orders is transparent in reduced-word language, where right weak order
is governed by initial subwords of reduced decompositions, whereas strong Bruhat
order is governed by arbitrary subwords 
\cite{BB05}. A complementary characterization, used throughout this paper, is that
$\sigma \wbo \omega$ is equivalent to containment of inversion sets of the inverses,
$\Inv(\sigma^{-1})\subseteq \Inv(\omega^{-1})$.  
This inversion viewpoint connects weak order to adjacent-swap sorting
procedures (such as bubble sort~\cite{KnuthTAOCP3}) and to earlier lattice-theoretic appearances of inversion containment~\cite{GR63}.  Bj\"orner’s 1984 study of orderings of Coxeter groups gave the first systematic
treatment of weak order, establishing basic structural properties \cite{Bjorner84Orderings}.  
A geometric model is provided by the
permutohedron: its one-skeleton encodes adjacent transpositions, and orienting edges by increasing
length yields the Hasse diagram of $(\sym{n},\wbo)$ \cite{Ziegler95}; see \cite{GR63} for a prototypical treatment. 

From a poset-theoretic viewpoint, the strong and weak Bruhat orders share the same natural grading by Coxeter length, yet they differ substantially in their global combinatorial behavior.  Weak order is a graded lattice, so meets and joins exist and one can exploit lattice-theoretic structure such as congruences and canonical quotients \cite{BB05,Reading04}.  Strong Bruhat order is not a lattice even in type $A$, but it has a remarkably rigid interval theory: Bruhat intervals are Eulerian and admit EL-shellings, yielding strong topological and enumerative consequences for their order complexes \cite{BjornerWachs82,BB05}.  At the same time, both orders satisfy strikingly strong antichain bounds of Sperner type.  For strong Bruhat order these follow from the geometry of the flag variety via Hard Lefschetz \cite{Stanley80}, whereas for weak order they rest on more delicate combinatorial constructions~\cite{GaetzGao20}.

\subsection{Comparability}\label{subsec:comparability}

A natural global question about a finite poset $(P,\preceq)$ is: if $X,Y$ are independent uniform
random elements of $P$, what is the probability that they are comparable? This \emph{comparability probability}, denoted $\Prob(X \preceq Y)$, is one of the most basic global statistics of a poset, alongside quantities such as
rank sizes, cover relations, Möbius invariants, and the topology of the order complex. In many
classical posets (for example, the Tamari lattice), the probability of comparability is remarkably tractable
and admits explicit formulas and refined enumerations \cite{Kreweras72,SCV_stanley86,Chapoton_tamari05,MFPR_mTamari11,Pittel97,Pittel99}. For the strong and weak Bruhat orders, by contrast, no exact formulas are known, and the literature does not even determine the leading order
asymptotics. Motivated by this gap, we study comparability in the strong and weak Bruhat orders on $\sym{n}$.

The probability of comparability can be framed in terms of the equivalent problem of interval enumeration. Let
\[
\weakintervals{n} \ :=\ \{ (\pi, \sigma)\in \sym{n}^2 : \pi \wbo \sigma\}
\qquad\text{and}\qquad
\strongintervals{n} \ :=\ \{ (\pi, \sigma)\in \sym{n}^2 : \pi \sbo \sigma\}
\]
denote the sets of intervals in the (right) weak and strong orders, respectively.  Thus for
independent uniform $\sigma_1,\sigma_2\in\sym{n}$,
\[
\Prob(\sigma_1\wbo\sigma_2)=\frac{|\weakintervals{n}|}{(n!)^2},
\qquad
\Prob(\sigma_1\sbo\sigma_2)=\frac{|\strongintervals{n}|}{(n!)^2}.
\]

We first determine the leading asymptotic scale of $\Prob(\sigma_1\wbo\sigma_2)$ with a sharp
leading constant.

\begin{theorem}\label{thm:main_weak}
As $n\to\infty$,
\[
\Prob(\sigma_1\wbo\sigma_2)
\ =\ \exp\Bigl(\bigl(-\tfrac12+o(1)\bigr)\,n\log n\Bigr).
\]
Equivalently,
\[
|\weakintervals{n}|
\ =\
(n!)^2\,\exp\Bigl(\bigl(-\tfrac12+o(1)\bigr)\,n\log n\Bigr).
\]

\end{theorem}

The lower bound in Theorem~\ref{thm:main_weak} is obtained by combining the Baik--Deift--Johansson
asymptotics for the longest increasing subsequence~\cite{BDJ99} with the Bochkov--Petrov lower
inequality for linear extensions in terms of antichain Greene--Kleitman--Fomin parameters~\cite{BP21}.
For the upper bound, we apply the Bochkov--Petrov upper inequality~\cite{BP21} and then condition on
the RSK shape: under Plancherel measure, $\Prob(\sigma_1\wbo\sigma_2)$ is bounded by an
average of an explicit Young-diagram functional $\Psi(\lambda)$.  The proof reduces to a deterministic uniform lower bound on $\Psi(\lambda)$, established via an inductive peeling argument.

Next, we first prove a quantitative lower bound on $\Prob(\sigma_1\sbo\sigma_2)$ showing
subexponential decay.

\begin{theorem}\label{thm:strong}
As $n\to\infty$,
\[
\Prob(\sigma_1\sbo\sigma_2)
\ \ge\
\exp(-(6+o(1))\sqrt{n} (\log^{3/2} n)).
\]
Equivalently,
\[
|\strongintervals{n}|
\ \ge\
(n!)^2\,\exp(-(6+o(1))\sqrt{n} (\log^{3/2} n)).
\] 
\end{theorem}

In particular, $\Prob(\sigma_1\sbo\sigma_2)$ decays strictly slower than $\exp(-cn)$ for all constant
$c>0$. The proof of Theorem~\ref{thm:strong} is constructive and given in Section~\ref{sec:strong}.
We exhibit two large structured families of permutations, such that a pair of permutations from their cartesian product is comparable in strong Bruhat order with probability $1-o(1)$. Comparability is verified via the tableau (Gale-order) criterion, which
reduces the required inequalities to a family of random-walk-type deviations controlled by concentration estimates. The number of pairs taken from the cartesian product of these structured families is a subexponential fraction of $(n!)^2$ and yields the stated lower bound.

Hammett and Pittel~\cite{HP08} initiated a systematic study of comparability probabilities in
Bruhat-type orders.  
For the weak order, writing $H(i)=\sum_{j\le i}\frac1j$, they obtained the bounds
\[
(n!)\prod_{i=1}^n H(i)
\ \le\
|\weakintervals{n}|
\ \le\
(n!)^2(0.362)^n.
\]
We note that these bounds are quite far apart; $\prod_{i=1}^{n} H(i) = \exp(O(n \log\log n))$ so these bounds imply that 
\begin{equation}
(1+o(1)) n \log n \leq \log |\mathcal{W}_n| \leq (2+o(1)) n \log n. \label{eqn:log}
\end{equation}
For the strong Bruhat order they proved
\[
(0.708)^n \ \le\ \Prob(\sigma_1\sbo\sigma_2)\ \le\ \frac{1}{n^2}.
\]

In this paper, we significantly improve three of the four bounds of Hammett--Pittel. For the weak order we show that the truth is, in a sense, exactly in the middle of Hammett and Pittel's bounds.  We determine precisely the asymptotics of $\log |\mathcal{W}_n|$, showing that \[\log |\mathcal{W}_n| = \left(\frac{3}{2}+o(1)\right)n \log n,\] exactly in the middle of \eqref{eqn:log}.  For the strong order, we improve the lower bound from exponential decay (\emph{i.e.}, $(0.708)^n$) to subexponential 
(of the form $\exp(-O(\sqrt{n \log^{3 n}}))$).
Despite numerical experiments from Hammett and Pittel suggesting that both of their upper bounds were qualitatively close to the true value, we show that the probability comparability for the weak order, governed by RSK shapes under the Plancherel measure, exhibits superexponential decay.

\subsection{Organization}
Section~\ref{sec:background} collects background and notation.  This includes the analytic estimates (Stirling, McDiarmid, the log-sum inequality), the characterizations of weak and strong Bruhat order, the RSK/Plancherel framework, and the poset-theoretic reduction expressing weak-order comparability in terms of linear extensions of permutation posets.  Section~\ref{sec:weak} proves the weak-order asymptotics (Theorem~\ref{thm:main_weak}).  The lower bound is obtained in Section~\ref{subsec:weak_lower} by combining typical longest-increasing-subsequence behavior with the Bochkov--Petrov lower inequality.  The upper bound is developed in two steps.  First, in Section~\ref{subsec:upper_to_Psi} we rewrite the probability as a Plancherel average and reduce the problem to a uniform exponent bound for a Young-diagram functional.  This we prove in Section~\ref{subsec:Psi_reduction}, by an inductive argument. In Section~\ref{sec:strong} we prove the strong-order lower bound (Theorem~\ref{thm:strong}).  The argument exhibits large structured families of permutations and verifies comparability with high probability via the tableau (Gale-order) criterion.

\subsection{Acknowledgements}

This collaboration began at the Graduate Research Workshop in Combinatorics (GRWC) 2025. The authors are grateful for the opportunity to collaborate despite coming from different areas of combinatorics, which would not have materialized without the environment and support provided by the GRWC. We also are grateful for productive conversations with Karimatou Djenabou, Alex Moon, and Sheila Sundaram in early stages of this project.

\section{Background}\label{sec:background}
In this section we collect the requisite background material needed to prove our bounds.

\subsection{Analytic Tools}
We begin by fixing notation for asymptotic comparisons; we follow standard conventions (see, e.g., \cite{GKP94}).

\begin{definition}\label{def:bigOsmallonotation}
Let $A\subset\mathbb{R}$ be such that $A\cap\mathbb{R}_{>0}$ is unbounded. (In practice, this will be an unbounded subset of the natural numbers.) Let $f,g \colon A \to \mathbb{R}$, and assume $g(n)>0$ for all $n\in A$. 

\begin{itemize}
    \item We write $f(n)=O(g(n))$ if there exists a constant $C>0$ and $n_0 \in \mathbb{N}$ such that
\[
|f(n)| \le C g(n) \quad \text{for all } n\in A \text{ with } n \ge n_0.
\]

\item We write $f(n)=o(g(n))$ if $\lim_{n\to\infty} \frac{f(n)}{g(n)} = 0$, where the limit is taken along $n\in A$.

\item We write $f(n)=\Theta(g(n))$ if there exists a constant $C>0$ and $n_0 \in \mathbb{N}$ such that
\[
\frac{1}{C}g(n)\le |f(n)| \le C g(n) \quad \text{for all } n\in A \text{ with } n \ge n_0.
\]
Equivalently, $f(n)=\Theta(g(n))$ if and only if $f(n)=O(g(n))$ and $g(n)=O(f(n))$. 

\item We write $f(n)\sim g(n)$ if $\lim_{n\rightarrow\infty} \frac{f(n)}{g(n)}=1$, where the limit is taken along $n\in A$.
\end{itemize}
\end{definition}

We interpret asymptotic terms inside algebraic expressions by algebraically isolating the asymptotic notation. For instance, if $f,g,h \colon A \to \mathbb{R}$ with $h(n)>0$ for all $n\in A$, then the statement $f(n)=g(n)+o(h(n))$ means $f(n)-g(n)=o(h(n))$. In particular, the statements $f(n) \sim g(n)$, $f(n)=(1+o(1))g(n)$, and $f(n)=(1-o(1))g(n)$ are equivalent.

We use the conventions $0! \deftobe 1$ and $0\log 0 \deftobe 0$. Explicitly, we extend $x\mapsto x\log x$ continuously at $x=0$.

We now discuss a fundamental asymptotic approximation that we will make heavy use of in this paper. \textit{Stirling's approximation} gives an asymptotic bound on the growth rate of the factorial function. 

\begin{theorem}[Stirling \cite{Stirling1730}] \label{thm:Stirlingapproximation}
    For $n\in\mathbb{N}$, we have \[n!\sim \sqrt{2\pi n}\left(\frac{n}{e}\right)^n.\]
\end{theorem}

\begin{remark}\label{rem:weakerstirlingapproximation}
It follows from Theorem~\ref{thm:Stirlingapproximation} that
\[
\log(n!)\ =\ n\log n - n + O(\log n)
\ =\ n\log n - (1+o(1))\,n.
\]
Equivalently,
\[
n!\ =\ \exp\!\hspace{.28mm}\bigl(n\log n - (1+o(1))\,n\bigr).
\]

\end{remark}

In fact, a classical related result guarantees the following crude version of Stirling's approximation due to \cite{R55}.

\begin{theorem}\label{thm:crude-stirling}
    
For all $n\in \mathbb{N}$ including $n=0$, we have 

\[\log(n!)\geq n\log n-n.\]

\end{theorem}

In the proof of Lemma~\ref{lem:strongLB}, we will make use of the following concentration inequality of McDiarmid to show some statistic does not deviate too far from its expectation. 

\begin{proposition}[McDiarmid \cite{M89}] \label{prop:mcdiarmid}
 Suppose $X_1, X_2, \dots, X_T$ are independent random variables taking values in $\mathcal{X}_i$, and $f:\mathcal{X}_1 \times \dots \times \mathcal{X}_T \to \mathbb{R}$ is such that for all inputs and all $i\in[T]$ \[
\big|f(x_1, \dots, x_{i-1}, x_i, x_{i+1}, \dots, x_T) - f(x_1, \dots, x_{i-1}, \hat{x}_i, x_{i+1} \dots x_T)\big| \leq c_i,\]
that is, changing the $i$th coordinate changes the value of $f$ by at most $c_i$.  Then
\[
\Prob\big(f(X_1, \dots, X_T) \leq  \e[f(X_1, \dots, X_T)] - \lambda \big) \leq \exp\left(- \frac{2\lambda^2}{\sum c_{i}^2} \right) 
\]

\end{proposition}

In the proof of Theorem~\ref{thm:Psi}, we use the log-sum inequality, which is a useful corollary of Jensen's theorem~\cite{Jensen06} frequently used in information theory.  

\begin{proposition}[Log-sum inequality \cite{CT91}] \label{prop:logsum}
Suppose $a_1, a_2, \dots, a_n$ and $b_1, \dots, b_n$ are non-negative numbers.  Then \[
\sum_{i=1}^{n} a_i \log \frac{a_i}{b_i} \geq \left(\sum_{i=1}^n a_i\right)  \log \left( \frac{\sum a_i}{\sum b_i} \right).
\]	
\end{proposition}

Finally, we record one other simple fact that we also use in the proof of Theorem \ref{thm:Psi} below.

\begin{proposition} \label{prop:sumi} For all integers $n \geq 1$,  
\[	\sum_{i=1}^{n} \frac{1}{\sqrt{i}} \leq \int_{0}^{n} \frac{1}{\sqrt{x}}dx = 2\sqrt{n}. \] 
\end{proposition} 
This follows by observing that $\sum_{i=1}^{n} i^{-1/2}$ is a right-hand Riemann sum for the integral of the decreasing function $x^{-1/2}$.  

\subsection{Bruhat Orders}

For a detailed introduction to weak and strong Bruhat orders on $\sym{n}$ (and their Coxeter group context), see Bj\"orner--Brenti~\cite{BB05}. Let $\sym{n}$ denote the symmetric group on $n$ letters, generated by the simple transpositions
\[
s_i = (i \;\; i+1), \quad 1 \le i \le n-1.
\]
For $\sigma \in \sym{n}$, the \emph{length} $\ell(\sigma)$ is the minimal number of simple transpositions needed to express $\sigma$ as a product. An expression of $\sigma$ as a product of $\ell(\sigma)$ simple transpositions is called a \emph{reduced expression}.

Both weak and strong Bruhat order are partial orders on $\sym{n}$ defined in terms of expressions for permutations. 

\begin{definition}\label{weak-bruhat-order}
Let $\sigma,\omega \in \sym{n}$. We say that $\sigma$ is less than or equal to $\omega$ in the \textit{(right) weak Bruhat order}, denoted $\sigma \wbo \omega$, if there exists a reduced expression
\[
\omega = s_{i_1}s_{i_2}\cdots s_{i_k}
\]
such that $\sigma = s_{i_1}s_{i_2}\cdots s_{i_m}$ for some $m \le k$.
\end{definition}

In other words, $\sigma$ is obtained as an \textit{initial segment} of a reduced expression for $\omega$. We note that there is a natural dual notion of \textit{left weak order} defined using final segments instead of initial segments, but we will not use it in this paper. The strong Bruhat order relaxes the prefix condition by allowing more general subexpressions.

\begin{definition}\label{def:strong-bruhat-order}
Let $\sigma,\omega \in \sym{n}$. We say that $\sigma$ is less than or equal to $\omega$ in the
\textit{strong Bruhat order}, denoted $\sigma \sbo \omega$, if there exists a reduced expression
\[
\omega = s_{i_1}s_{i_2}\cdots s_{i_k}
\]
and a subset of indices $1\le j_1<\cdots< j_m\le k$ such that
\[
\sigma = s_{i_{j_1}}s_{i_{j_2}}\cdots s_{i_{j_m}}.
\]
Equivalently (although not obviously so), there exists such a subsequence of indices for every reduced expression of $\omega$.
\end{definition}

Thus, while weak order allows only prefixes of reduced expressions, strong Bruhat order allows arbitrary subexpressions. Accordingly, weak Bruhat order is a coarser order on $\sym{n}$.

We now continue with a different combinatorial characterization for each of these Bruhat orders that does not use expressions. The following is well-established, with a direct proof appearing in \cite{BW91}. 

\begin{theorem}\label{thm:inversion-set-for-weak-order}
For $\sigma,\omega \in \sym{n}$,
\[
\sigma \wbo \omega
\quad \text{if and only if} \quad
\mathrm{Inv}(\sigma^{-1}) \subseteq \mathrm{Inv}(\omega^{-1}),
\]
where
\[
\mathrm{Inv}(\sigma) = \{(i,j) : i<j,\ \sigma(i)>\sigma(j)\}
\]
denotes the inversion set of a permutation $\sigma$.
\end{theorem}

This characterization makes weak order particularly transparent combinatorially: moving upward in the weak order corresponds to adding inversions in the inverse.

One classical formulation of the strong Bruhat order is the tableau criterion, which relates strong Bruhat order to an order on sets called Gale order. While the relation to tableaux is obscured in our presentation, the tableau criterion is equivalent to element-wise comparison of a judiciously chosen pair of semistandard Young tableaux, as detailed in \cite{BB05}.

\begin{definition}\label{def:Gale-order}
Let $A,B \subseteq \mathbb{Z}$ be finite subsets with $|A|=|B|=k$. Write
\[
A = \{a_1 < a_2 < \cdots < a_k\}, \qquad
B = \{b_1 < b_2 < \cdots < b_k\}.
\]
We say that $A$ is \emph{less than or equal to $B$ in Gale order}, and write
\[
A \gale B,
\]
if
\[
a_i \le b_i \quad \text{for all } i = 1,\dots,k.
\]
\end{definition}

\begin{theorem}[Tableau criterion, \cite{BB05}]\label{thm:tableau-criterion}
Let $\sigma,\omega \in \sym{n}$. Then
\[
\sigma \sbo \omega
\]
if and only if for each $k\in [n]$, we have 

\[\{\sigma(1),\ldots,\sigma(k)\}\gale \{\omega(1),\ldots,\omega(k)\}.\]
\end{theorem}

\begin{remark}
Throughout this paper, we rely on Theorems~\ref{thm:inversion-set-for-weak-order} and \ref{thm:tableau-criterion} and avoid working directly with expressions and subexpressions. Nevertheless, the expression-based perspective clarifies why weak and strong Bruhat order are so closely related and why it is natural to discuss them together. Moreover, this viewpoint makes clear how both orders generalize to Bruhat orders of other geometric types: one simply replaces $\sym{n}$ with the appropriate Weyl group. One can give analogues of Theorems~\ref{thm:inversion-set-for-weak-order} and \ref{thm:tableau-criterion} for other types, which are just the combinatorial outcome of this natural generalization using Weyl groups.
  
\end{remark}

\subsection{Permutations and Tableau}

In this subsection, we review the Robinson--Schensted--Knuth correspondence and several of its consequences that will play a central role throughout this paper. In particular, we recall how RSK associates to each permutation in $\sym{n}$ a partition of $n$ encoding detailed combinatorial information, and we describe how this information interacts with probabilistic structures.

A \textit{partition} $\lambda$ of $n \in \mathbb{N}$, denoted $\lambda \vdash n$, is a finite sequence of positive integers
\[
\lambda_1 \ge \lambda_2 \ge \cdots \ge \lambda_k
\]
such that
\[
\sum_{i=1}^k \lambda_i = n.
\]
The integers $\lambda_i$ are called the \emph{parts} of the partition.

We will rely on the following bound on the number of partitions of $n$, due to Hardy--Ramanujan.

\begin{theorem}[\cite{HR18}]\label{thm:Hardy-Ramanujan}
    Let $p(n)=|\{\lambda\mid\lambda\vdash n\}|$ be the number of partitions of $n$. Then, 
\[p(n)=\exp\left(O(\sqrt{n})\right).\]
\end{theorem}

To each partition $\lambda = (\lambda_1,\ldots,\lambda_k)$ we associate a \textit{Young diagram}, consisting of left-justified rows of $1 \times 1$ boxes, with $\lambda_j$ boxes in the $j$th row from the top. Given $\lambda\vdash n$, its \emph{conjugate} $\lambda'$ is the partition defined by
\[
\lambda'_j \coloneqq \#\{i:\ \lambda_i\ge j\}.
\]
Equivalently, the Young diagram of $\lambda'$ is the transpose of the Young diagram of $\lambda$. A \textit{standard Young tableau} of shape $\lambda$ is a filling of the boxes of the Young diagram of $\lambda$ with the integers $\{1,2,\dots,n\}$ such that the entries increase strictly along each row and down each column.

\begin{example}\label{ex:SYT}
Let $\lambda=(4,2,1)\vdash 7$, so $\conj{\lambda}=(3,2,1,1)$.  The Young diagram of $\conj{\lambda}$
and an example of a standard Young tableau of shape $\lambda$ are:
\[
\scalebox{0.85}{$\ydiagram{3,2,1,1}$}
\qquad\text{and}\qquad
\scalebox{0.85}{$
\begin{ytableau}
1 & 2 & 4 & 7 \\
3 & 5 \\
6
\end{ytableau}\,\,.
$}
\]
\end{example}

We write $\SYT(\lambda)$ for the set of standard Young tableaux of shape $\lambda$. The number of standard Young tableaux of a given shape $\lambda$ is well studied. We define some relevant quantities before giving the formula. For a box in the $i^\text{th}$ row (with the top row being row $1$) and $j^\text{th}$ column (with the leftmost column being column $1$) of $\lambda$, we define the hook length $\hook{\lambda}{i}{j} \deftobe \lambda_i-j+\conj{\lambda}_j-i+1$. This is the number of boxes to the right of $(i,j)$ in its row or below $(i,j)$ in its column, including $(i,j)$ itself. 

\begin{definition}\label{def:hook-length}
    The \textbf{hook length product} of a partition $\lambda$ is \[\hookproduct{\lambda}=\prod_{(i,j)\in\lambda}\hook{\lambda}{i}{j},\]
    with the product taken over all boxes of $\lambda$.    
\end{definition}

We record here a simple observation about the products of hook lengths.
\begin{lemma}\label{lem:row_col_prod} 
Let $\lambda \vdash n$. Then the hook length
\[
h_{\lambda}(i,j) \geq \min\{ \lambda_i-j+1, \lambda_j'-i+1\}.
\]
In particular, the product of hook lengths in the $i$th row is at least $(\lambda_i)!$ and the product of hook lengths in the $j$th column is at least $(\lambda_j')!.$
\end{lemma}

\begin{theorem}\label{thm:hook-length-formula}
    Let $\lambda$ be a partition of $n$ and let $f^\lambda$ be the number of standard Young tableaux of shape $\lambda$. Then, 

    \[f^\lambda=\frac{n!}{\hookproduct{\lambda}}.\]
\end{theorem}

Having established some basic properties of Young diagrams and tableaux, we now turn to the key result relating them to permutations. The Robinson--Schensted--Knuth correspondence associates to each permutation a pair of \textit{standard Young tableaux} of the same shape.

\begin{theorem}[RSK correspondence]\label{thm:RSK}
There is a bijection
\[
\RSK \colon \sym{n} \longrightarrow \bigsqcup_{\lambda \vdash n}
\SYT(\lambda) \times \SYT(\lambda),
\]
which sends a permutation $\omega \in \sym{n}$ to a pair $(P(\omega),Q(\omega))$ of standard Young tableaux of the same shape $\lambda(\omega)$.
\end{theorem}

The combinatorial details of the construction are not relevant for what follows. What is relevant is the map taking $\omega$ to $\lambda(\omega)$. Notably, the shape $\lambda(\omega)$ reflects the structure of increasing subsequences of $\omega$, when $\omega$ is written in one-line notation.

\begin{theorem}[\cite{Sch61}]\label{thm:longest-inc-subsequence}
Let $\omega \in \sym{n}$, and let $\lambda(\omega) = (\lambda_1,\lambda_2,\dots)$ be the shape of its RSK tableaux. Then $\lambda_1$ is the length of the longest increasing subsequence of $\omega$.
\end{theorem}

Thus, the first row of the RSK shape records the size of the longest increasing pattern in the permutation. Greene's theorem extends this observation by relating the lengths of other rows of $\lambda(\omega)$ to increasing subsequences in $\omega$.

\begin{theorem}[Greene's Theorem~\cite{G74}]\label{thm:Greene}
Let $\omega \in \sym{n}$, and let $\lambda(\omega) = (\lambda_1,\lambda_2,\dots)$ be the shape of its RSK tableaux. For each $k \ge 1$, the sum
\[
\lambda_1 + \cdots + \lambda_k
\]
equals the maximum total length of a union of $k$ disjoint increasing subsequences of $\omega$.
\end{theorem}

Equivalently, the row lengths of $\lambda(\omega)$ encode the optimal way to decompose $\omega$ into increasing subsequences. 

\begin{example}\label{ex:Greene's-theorem}
Let $\omega=3\,2\,1\,6\,5\,4\,7\in\sym{7}$, which has $\lambda(\omega)=(3,2,2)$ under RSK.
We verify Greene's theorem for $\omega$ by exhibiting non-unique optimal unions of $k$ disjoint increasing subsequences:
\[
\begin{array}{rll}
k=1: & (1,6,7), & \text{total length }3=\lambda_1,\\[2pt]
k=2: & (1,6,7)\ \cup\ (2,5), & \text{total length }5=\lambda_1+\lambda_2,\\[2pt]
k=3: & (1,6,7)\ \cup\ (2,5)\ \cup\ (3,4), & \text{total length }7=\lambda_1+\lambda_2+\lambda_3.
\end{array}
\]
\end{example}

We rely heavily on the RSK shape $\lambda(\omega)$ of a permutation $\omega\in\sym{n}$. We next discuss the probability measure induced on partitions of $n$ when we choose a permutation uniformly at random from $\sym{n}$. The resulting measure is called the \textit{Plancherel measure}.

\begin{definition}\label{def:Plancherel}
For a partition $\lambda \vdash n$, the Plancherel measure is defined by
\[
\planch(\lambda) = \frac{(f^\lambda)^2}{n!}=\frac{n!}{\hookproduct{\lambda}^2},
\]
where the second equality follows from Theorem~\ref{thm:hook-length-formula}.
\end{definition}

Since the RSK correspondence is a bijection between $\sym{n}$ and pairs of standard Young tableaux of
the same shape, the RSK shape of a uniform random permutation is distributed according to Plancherel
measure (see, e.g., \cite{S99}):

\begin{proposition}\label{prop:Plancherel}
Consider a permutation $\sigma\in\sym{n}$ selected uniformly at random. Fix a Young diagram $\lambda^*$. Then, we have $\lambda(\sigma)=\lambda^*$ with probability $\planch(\lambda^*)$.  
\end{proposition}

The Plancherel measure will be a central tool for determining bounds on the number of intervals in weak Bruhat order. Further, since $\planch(\lambda)=\frac{n!}{\hookprod(\lambda)^2}\le 1$ for every $\lambda\vdash n$, we immediately obtain a uniform lower bound on $\hookprod(\lambda)$.

\begin{lemma}\label{lem:hook-uniform}
For all $n\in\mathbb N$ and all $\lambda\vdash n$,
\[
2\log \hookprod(\lambda)\ \ge\ n\log n - n.
\]
\end{lemma}

\begin{proof}
For each $\lambda\vdash n$,
\[
\planch(\lambda)
=\frac{(f^\lambda)^2}{n!}
=\frac{n!}{\hookprod(\lambda)^2}\leq 1
\]
Using the crude Stirling lower bound $\log(n!)\ge n\log n-n$, we conclude
\[
2\log\hookprod(\lambda)\ \ge\ \log(n!) \ge  n\log n-n,
\]
so the claim holds.
\end{proof}

We can also give refined probabilistic descriptions of the length of the longest increasing subsequence of $\omega\in\sym{n}$ or, equivalently by Theorem~\ref{thm:longest-inc-subsequence}, of $\lambda(\omega)_1$. 

We will use the following probability bound on the first row of the RSK shape of a uniform random permutation, which follows immediately from the Baik--Deift--Johansson asymptotics for the longest increasing subsequence.

\begin{lemma}\label{lem:LIS_3sqrtn}
Let $n\in\mathbb N$, and let $\omega$ be chosen uniformly at random from $\sym{n}$.  Let $\lambda(\omega)$ denote the RSK shape of $\omega$. Then
\[
\Prob(\lambda(\omega)_1\le 3\sqrt{n})\;=\;1-o(1).
\]
\end{lemma}

\begin{proof}
By Theorem~\ref{thm:longest-inc-subsequence}, $\lambda(\omega)_1$ is the length of the longest increasing subsequence in $\omega$.
Accordingly, by \cite[Theorem~1.1]{BDJ99}, the random variable
\[
X_n \;:=\; \frac{\lambda(\omega)_1-2\sqrt{n}}{n^{1/6}}
\]
converges in distribution (as $n\to\infty$) to the Tracy--Widom GUE law.

Moreover, writing $F_n(t):=\Prob(X_n\le t)$ for the distribution function of $X_n$, \cite[Equation (1.8)]{BDJ99} gives an explicit upper-tail bound: for $M>0$ sufficiently large there exist constants $c>0$ and $C(M)>0$ such that
\[
1-F_n(t)=\Prob(X_n>t)\ \le\ C(M)e^{-c t^{3/5}}
\]
for all $t$ in the range $M\le t\le n^{5/6}-2n^{1/3}$.

Taking $t:=n^{1/3}$, we have $t\to\infty$ and also $t\le n^{5/6}-2n^{1/3}$ for all sufficiently large $n$, so the bound applies and yields
\[
\Prob(X_n>n^{1/3})\ \le\ C(M)e^{-c n^{1/5}} \ =\ o(1).
\]
Finally, $X_n>n^{1/3}$ if and only if $\lambda(\omega)_1>2\sqrt{n}+n^{1/6}\cdot n^{1/3}=3\sqrt{n}$, so $\Prob(\lambda(\omega)_1>3\sqrt{n})=o(1)$, equivalently $\Prob(\lambda(\omega)_1\le 3\sqrt{n})=1-o(1)$.
\end{proof}

\subsection{Poset-theoretic results}

In this section, we recall some useful results about chains and antichains in finite posets, and apply them to the permutation poset associated to a given permutation $\sigma\in\sym{n}$. The goal is to show that the weak Bruhat order is closely related to \textit{linear extensions} of posets.

\begin{definition}\label{def:linear-extension}
A \textbf{linear extension} of a finite poset $(P,\prec)$ is a total order $\prec'$ on $P$ such that for all $p_1,p_2\in P$,
\[
p_1\prec p_2\ \Longrightarrow\ p_1\prec' p_2.
\]
We write $e(P)$ for the number of linear extensions of $P$.
\end{definition}

\begin{example}\label{ex:poset-extension}
Let $(P,\prec)$ be the poset on $\{a,b,c,d\}$ with relations $a\prec b$ and $a\prec c$ (and no others).
Define a total order $\prec'$ on the same ground set by declaring
\[
a\prec' b,\qquad b\prec' c,\qquad c\prec' d.
\]
Then $\prec'$ extends $\prec$ (since $a\prec' b$ and, by transitive closure, $a\prec' c$), and hence $(P,\prec')$ is a linear extension of $(P,\prec)$.
\end{example}

\begin{definition}\label{def:chains-and-antichains}
Let $(P,\prec)$ be a poset. A \emph{chain} is a subset $C\subseteq P$ such that $\prec$ restricts to a total order on~$C$, \emph{i.e.}, for all distinct $c_1,c_2\in C$ we have either $c_1\prec c_2$ or $c_2\prec c_1$. 
An \emph{antichain} is a subset $C\subseteq P$ such that no two distinct elements are comparable, \emph{i.e.}, for all distinct $c_1,c_2\in C$ we have neither $c_1\prec c_2$ nor $c_2\prec c_1$.
\end{definition}

For a finite poset $P$ on $n$ elements and $k\ge 1$, let
\[
A_k(P)\ :=\ \max\{|U|:\ U\subseteq P\ \text{can be covered by $k$ antichains}\},
\]
and
\[
C_k(P)\ :=\ \max\{|U|:\ U\subseteq P\ \text{can be covered by $k$ chains}\}.
\]
Set $A_0(P)=C_0(P)=0$ and define the Greene--Kleitman--Fomin parameters by
\[
a_k(P)\ :=\ A_k(P)-A_{k-1}(P),
\qquad
c_k(P)\ :=\ C_k(P)-C_{k-1}(P)
\qquad (k\ge 1).
\]

It immediately follows that

\[
\sum_{k\ge 1} a_k(P)=\sum_{k\ge 1} c_k(P)=n.
\]

\begin{proposition}[\cite{Greene76}]
Let $P$ be a finite poset. For $k\ge 1$, one has $a_{k+1}(P)\le a_k(P)$ and $c_{k+1}(P)\le c_k(P)$. In particular, $(a_k(P))_{k\ge 1}$ and $(c_k(P))_{k\ge 1}$ are partitions of $n$.
\end{proposition}

While the above results are illuminating in their own right, Greene--Kleitman--Fomin parameters also find widespread application. In this paper, we will use them to obtain a two-sided bound on the number of linear extensions of a given poset, a statistic whose computational complexity is $\#P$-complete. Bochkov and Petrov proved the following:

\begin{theorem}[{\cite{BP21}}]
\label{thm:BP}
Let $P$ be a finite poset on $n$ elements, with antichain and chain Greene--Kleitman--Fomin parameters $a_i(P)$ and $c_j(P)$, respectively.  Then
\[
\prod_i a_i(P)!\ \le\ e(P)\ \le\ \frac{n!}{\prod_j c_j(P)!}.
\]
\end{theorem}

We give another result relating chains with antichains, due to Mirsky. This is the dual statement to the well-know Dilworth's theorem.

\begin{theorem}[Mirsky's Theorem~\cite{M71}]\label{thm:Mirsky-theorem}
Let $P$ be a finite poset. The minimum number of \emph{antichains} needed to partition $P$ is equal
to the size of a largest \emph{chain} in $P$ (namely, the height of $P$).
\end{theorem}

\subsection{The permutation poset}\label{subsec:permutation_poset}

We now introduce the permutation poset and explain how it reduces weak-order comparability to estimating the number of linear extensions.

\begin{definition}\label{def:permutation-poset}
Let $\sigma\in\sym{n}$. The \emph{permutation poset} $\permposet{\sigma}$ is the poset on ground set
$[n]:=\{1,\dots,n\}$ with order relation
\[
i\prec j
\quad\Longleftrightarrow\quad
i<j\ \text{ and }\ \sigma(i)<\sigma(j).
\]
\end{definition}

For a permutation $\omega\in\sym{n}$, $\permposet{\omega}$ is a
linear extension of $\permposet{\sigma}$ if and only if
\[
\Inv(\omega)\subseteq \Inv(\sigma),
\]
equivalently $\omega^{-1}\wbo \sigma^{-1}$ by Theorem~\ref{thm:inversion-set-for-weak-order}. This equivalence is standard and admits a short direct verification from the definitions; see
Bj\"orner--Wachs~\cite[Proposition~4.1 and Section~6]{BW91} for a general weak-order-to-linear-extension framework. Taking inverses, we obtain
\[
\sigma_1 \wbo \sigma_2
\quad\Longleftrightarrow\quad
\Inv(\sigma_1^{-1})\subseteq \Inv(\sigma_2^{-1})
\quad\Longleftrightarrow\quad
\sigma_1^{-1}\ \text{is a linear extension of}\ \permposet{\sigma_2^{-1}}.
\]
Consequently, for every fixed $\sigma_2\in\sym{n}$,
\[
\#\{\sigma_1\in\sym{n}:\ \sigma_1\wbo\sigma_2\}
\ =\
e\!\big(\permposet{\sigma_2^{-1}}\big),
\qquad
\Prob(\sigma_1\wbo\sigma_2\mid \sigma_2)
\ =\ \frac{e(\permposet{\sigma_2^{-1}})}{n!}.
\]
Averaging over uniform $\sigma_2$ and reindexing the sum by inversion (a bijection of $\sym{n}$) yields the
basic reduction that weak-order comparability is governed by the typical size of
$e(\permposet{\sigma})$:
\begin{equation}\label{eq:weak_prob_extensions}
\Prob(\sigma_1\wbo\sigma_2)
\ =\
\frac{1}{(n!)^2}\sum_{\sigma\in\sym{n}} e\!\big(\permposet{\sigma}\big).
\end{equation}

\section{The weak order}\label{sec:weak}
In this section we prove Theorem~\ref{thm:main_weak} by matching lower and upper bounds for
$\Prob(\sigma_1\wbo\sigma_2)$ via the extension-count formula \eqref{eq:weak_prob_extensions}.
For the lower bound, we restrict to permutations with typical longest increasing subsequence and apply
the Bochkov--Petrov lower inequality. For the upper bound, we condition on RSK shape, apply the
Bochkov--Petrov upper inequality in terms of Greene--Kleitman--Fomin parameters, and reduce the problem
to a deterministic uniform lower bound on an exponent functional $\Psi(\lambda)$, proved by a
row-peeling reduction to a fixed-width strip and a direct strip analysis.

\subsection{Lower bound via typical \texorpdfstring{$\mathrm{LIS}$}{LIS} and the Bochkov--Petrov lower inequality}\label{subsec:weak_lower}

We now prove the lower bound in Theorem~\ref{thm:main_weak}. By \eqref{eq:weak_prob_extensions}, it
suffices to lower-bound $\sum_{\sigma\in\sym{n}} e(\permposet{\sigma})$.

\begin{proposition}\label{prop:weak_lower_bound}
As $n\to\infty$,
\[
\Prob(\sigma_1\wbo\sigma_2)
\ \ge\
\exp\Bigl(\bigl(-\tfrac12+o(1)\bigr)\,n\log n\Bigr).
\]
\end{proposition}

\begin{proof} 
Set $t:=\lceil 3\sqrt{n}\rceil$ and define
\[
T\ :=\ \{\sigma\in \sym{n}:\ \text{the longest increasing subsequence of $\sigma$ has length at most } t\}.
\]
By Lemma~\ref{lem:LIS_3sqrtn}, the longest increasing subsequence of a uniform random $\sigma\in\sym{n}$
has length at most $3\sqrt{n}$ with probability $1-o(1)$, and since $t\ge 3\sqrt{n}$ this implies
$\Prob(\sigma\in T)=1-o(1)$. Therefore
\[
|T|\ =\ (1-o(1))\,n!.
\]

Fix $\sigma\in T$. Any chain in $\permposet{\sigma}$ is an increasing subsequence of $\sigma$, so
the height of $\permposet{\sigma}$ is at most $t$. By Mirsky's theorem
(Theorem~\ref{thm:Mirsky-theorem}), $\permposet{\sigma}$ can be partitioned into $t$ antichains. Equivalently, the antichain Greene--Kleitman--Fomin parameters $a_1(\permposet{\sigma}),\dots,a_t(\permposet{\sigma})$ satisfy
\[
a_1(\permposet{\sigma})+\cdots+a_t(\permposet{\sigma})\ =\ n.
\]
By the Bochkov--Petrov lower bound (Theorem~\ref{thm:BP}),
\[
e(\permposet{\sigma})\ \ge\ \prod_{i=1}^t a_i(\permposet{\sigma})!.
\]

Write $n=qt+r$ with $q:=\lfloor n/t\rfloor$ and $0\le r<t$.  We claim that among all integer
$t$-tuples $(a_1,\ldots,a_t)$ with $\sum_{i=1}^t a_i=n$, the product $\prod_{i=1}^t a_i!$ is minimized
when the entries are as equal as possible, \emph{i.e.}, when $t-r$ of them equal $q$ and $r$ of them equal
$q+1$.  Indeed, suppose $a>b+1$ are two entries. Then
\[
\frac{(a-1)!(b+1)!}{a!\,b!}=\frac{b+1}{a}\le 1,
\]
so
\[
(a-1)!(b+1)!\ \le\ a!\,b!.
\]
Thus, replacing the pair $(a,b)$ by $(a-1,b+1)$ (which preserves the sum) does not increase the
product of factorials, and strictly decreases it unless $a=b+1$. Iterating this balancing operation
until no pair differs by more than $1$, we reach the configuration with $t-r$ entries equal to $q$
and $r$ entries equal to $q+1$, which therefore minimizes $\prod_{i=1}^t (a_i!)$. Since the entries are nonnegative integers with fixed sum $n$, repeated balancing must terminate after finitely many steps (for instance, $\sum_{i=1}^t a_i^2$ strictly decreases at each step). Consequently,
\[
\prod_{i=1}^t a_i!\ \ge\ (q!)^{\,t-r}\,\big((q+1)!\big)^{\,r}.
\]

Therefore
\[
\sum_{\sigma\in\sym{n}} e(\permposet{\sigma})
\ \ge\ \sum_{\sigma\in T} e(\permposet{\sigma})
\ \ge\ (1-o(1))\,n!\,(q!)^{\,t-r}\,\big((q+1)!\big)^{\,r}.
\]

Plugging this into \eqref{eq:weak_prob_extensions} gives
\[
\Prob(\sigma_1\wbo\sigma_2)
\ \ge\ (1-o(1))\,\frac{(q!)^{\,t-r}\,\big((q+1)!\big)^{\,r}}{n!}.
\]
Taking logs and applying Theorem~\ref{thm:Stirlingapproximation} (Stirling's approximation) yields
\[
\log\big((q!)^{\,t-r}\,\big((q+1)!\big)^{\,r}\big)
\ =\ \Bigl(\tfrac12+o(1)\Bigr)\,n\log n\ +\ O(n),
\]
using $t=\Theta(\sqrt n)$ and hence $q=\Theta(\sqrt n)$. Consequently,
\[
\log \Prob(\sigma_1\wbo\sigma_2)
\ \ge\
-\Bigl(\tfrac12+o(1)\Bigr)\,n\log n,
\]
which is equivalent to the claimed bound.
\end{proof}

\subsection{Upper bound via shape averaging}\label{subsec:upper_to_Psi}

We now turn to the upper bound in Theorem~\ref{thm:main_weak}.  We combine the reduction
\eqref{eq:weak_prob_extensions} with the Bochkov--Petrov upper inequality and then average over the
RSK shape under Plancherel measure.

First, for every $\sigma_2\in\sym{n}$,
\[
\Prob(\sigma_1\wbo \sigma_2\mid \sigma_2)
= \frac{e\big(\permposet{\sigma_2^{-1}}\big)}{n!}.
\]
Conditioning on the RSK shape of $\sigma_2$ gives
\[
\Prob(\sigma_1\wbo \sigma_2)
\ =\ \sum_{\lambda\vdash n}\Prob(\lambda(\sigma_2)=\lambda)\cdot
 \Prob(\sigma_1\wbo \sigma_2\mid \lambda(\sigma_2)=\lambda).
\]

Fix $\sigma\in\sym{n}$ and write $\lambda=\lambda(\sigma)$ for its RSK shape.  By Theorem~\ref{thm:Greene} (Greene's theorem), together with the observation that chains in $\permposet{\sigma}$ are precisely increasing subsequences of $\sigma$ (and unions of $k$ chains correspond to unions of $k$ disjoint increasing subsequences), the Greene--Kleitman--Fomin chain functionals of $\permposet{\sigma}$ satisfy
\[
C_k\big(\permposet{\sigma}\big)\ =\ \lambda_1+\cdots+\lambda_k
\qquad (k\ge 1),
\]
and therefore the corresponding chain parameters are
\[
c_k\big(\permposet{\sigma}\big)
:= C_k\big(\permposet{\sigma}\big)-C_{k-1}\big(\permposet{\sigma}\big)
= \lambda_k.
\]
Applying the Bochkov--Petrov upper bound (Theorem~\ref{thm:BP}) to $\permposet{\sigma_2^{-1}}$ yields
\[
e\big(\permposet{\sigma_2^{-1}}\big)
\ \le\ \frac{n!}{\prod_{k\ge 1} c_k\big(\permposet{\sigma_2^{-1}}\big)!}
\ =\ \frac{n!}{\prod_{k\ge 1} \lambda_k!},
\]
so that
\[
\Prob(\sigma_1\wbo \sigma_2\mid \lambda(\sigma_2)=\lambda)
 \le\ \frac{1}{\prod_{k\ge 1} \lambda_k!}.
\]

By Proposition~\ref{prop:Plancherel} $\lambda(\sigma_2)$ is distributed according to Plancherel measure, and we obtain
\begin{align*}
\Prob(\sigma_1\wbo \sigma_2)
&\le \sum_{\lambda\vdash n}\planch(\lambda)\cdot \frac{1}{\prod_i \lambda_i!}\\
&= \sum_{\lambda\vdash n}\frac{n!}{\hookprod(\lambda)^2}\cdot \frac{1}{\prod_i \lambda_i!}\\
&= n!\sum_{\lambda\vdash n}\exp\!\Bigl(- 2\log\hookprod(\lambda) - \sum_i \log(\lambda_i!)\Bigr).
\end{align*}

We define the exponent functional
\[
\Psi(\lambda)
\coloneqq
2\log\hookprod(\lambda)+\sum_i\log(\lambda_i!).
\]
Then
\[
\Prob(\sigma_1\wbo \sigma_2)
\ \le\ n!\sum_{\lambda\vdash n} e^{-\Psi(\lambda)}.
\]

The main deterministic input is the following uniform lower bound for $\Psi(\lambda)$.

\begin{theorem}\label{thm:Psi}
For all $n \in \mathbb{N},$ and all  $\lambda\vdash n$,
\[
\Psi(\lambda)\ \ge\ \frac{3}{2}\,n\log n\ -\ 5 n.
\]
\end{theorem}

\begin{remark} 
We do not attempt to optimize the constant $5$ appearing above, instead preferring a clean statement.  We note, however, that the proof of Lemma \ref{lem:strip} below requires that the constant is at least $\frac{9}{2}.$
\end{remark}

Assuming Theorem~\ref{thm:Psi}, the upper bound in Theorem~\ref{thm:main_weak} is immediate.

\begin{proposition}
\label{prop:Psi-to-prob}
Assume Theorem~\ref{thm:Psi}.  Then
\[
\Prob(\sigma_1\wbo \sigma_2)
\ \le\
\exp\!\Bigl(-\tfrac12\,n\log n+O(n)\Bigr).
\]
\end{proposition}

\begin{proof}
By Theorem~\ref{thm:Psi}, for all $\lambda\vdash n$ and all $n$ one has
\[
\Psi(\lambda)\ \ge\ \tfrac32\,n\log n - 5 n.
\]

By Remark~\ref{rem:weakerstirlingapproximation}, $\log(n!)=n\log n-(1+o(1))\,n$, equivalently
$n! = \exp\!\bigl(n\log n-(1+o(1))\,n\bigr)$, and therefore
\begin{align*}
\Prob(\sigma_1\wbo \sigma_2)
\ &\le\ \sum_{\lambda\vdash n}\exp\!\Bigl(-\Psi(\lambda) + n\log n-(1+o(1))\,n\Bigr)\\
\ &\le\ \sum_{\lambda\vdash n}\exp\Bigl(-\tfrac12\,n\log n + O(n)\Bigr)\\
\ &=\ p(n)\cdot \exp\Bigl(-\tfrac12\,n\log n + O(n)\Bigr),
\end{align*}
where $p(n)$ is the number of partitions of $n$. Using Theorem~\ref{thm:Hardy-Ramanujan} (Hardy-Ramanujan),
$p(n)=\exp(O(\sqrt{n}))$, we obtain that,
\[
\Prob(\sigma_1\wbo \sigma_2)
\ \le\ \exp\Bigl(-\tfrac12 n\log n + O(n) + O(\sqrt{n})\Bigr)
\ =\ \exp\Bigl(-\tfrac12 n\log n + O(n)\Bigr).
\]
\end{proof}

We now turn toward proving Theorem~\ref{thm:Psi}.

\subsection{The proof of Theorem~\ref{thm:Psi}}\label{subsec:Psi_reduction}

The proof of Theorem~\ref{thm:Psi} proceeds in two estimates. First we show that the inequality 

\begin{equation} 
\Psi(\lambda) \geq \frac{3}{2} n \log n - 5 n. \label{eqn:psilb} 
\end{equation} 
holds for $\lambda$ with $\lambda_1 \leq \sqrt{n}$.  Then, we show that any partition with $\lambda_1 \geq \sqrt{n}$ cannot be a smallest counterexample to Theorem ~\ref{thm:Psi}.  These, combined, prove the result.  

For what follows, suppose $\lambda \vdash n$, where $s = \lambda_1$ denotes the number of columns of $\lambda$ and $t = \lambda_1'$ denote the number of rows. Let $\lambda_1, \lambda_2, \dots, \lambda_t$ denote the row lengths of $\lambda$ and $\lambda_1', \lambda_2' \dots, \lambda'_{s}$ denote the column lengths.

\begin{lemma} 
Suppose $\lambda \vdash n$, and $s = s(\lambda) \leq \sqrt{n}$.  
Then \[ 
	\Psi(\lambda) \geq \frac{3}{2} n \log n - 5 n. 
\] \label{lem:skinny}
\end{lemma} 

\begin{proof} 
To make the formulae slightly easier to read, we let $c_i = \lambda_i'$ be the length of the $i$th column of $\lambda$ and let $r_i=\lambda_i$ denote the length of the $i$th row.   

Applying Lemma \ref{lem:row_col_prod} to $\pi(\lambda)$ along columns, and using the crude Stirling inequality Theorem~\ref{thm:crude-stirling}, which says that $\log(c_i!) \geq c_i \log c_i -  c_i$, we obtain 
\begin{align*} 
\Psi(\lambda) = 2 \log(\pi(\lambda)) + \sum_{j=1}^{t} \log(r_j!) &\geq 2 \sum_{i=1}^{s} \log(c_i!) + \sum_{j=1}^{t} \log(r_j!) \\
&\geq \left( 2 \sum_{i=1}^{s}  c_i \log c_i + \sum_{j=1}^{t} \log(r_j!) \right) - 2n.
\end{align*} 
Here we used the fact that $\sum c_i = n$ in the last inequality.

Now, rearranging the terms in $\log(r_j!)$, note that 
\[
\sum_{j=1}^{t}\log(r_j!)
=\log\!\Big(\prod_{j=1}^{t} r_j!\Big)
=\log\!\Big(\prod_{i=1}^{s} i^{c_i}\Big)
=\sum_{i=1}^{s} c_i\log i,
\]
from which 
\begin{align*} 
\Psi(\lambda) &\geq \left( 2 \sum_{i=1}^{s}  c_i \log c_i + \sum_{j=1}^{s} c_i \log(i)  \right) - 2n \\
&= 2  \left( \sum_{i=1}^{s} c_i \log \left( \frac{c_i}{i^{-1/2}} \right)  \right) - 2n.
\end{align*}
We now apply the log-sum inequality (Proposition \ref{prop:logsum}) along with Proposition \ref{prop:sumi}  to observe that
\begin{align*}
\Psi(\lambda) \geq 2 \left( \sum_{i=1}^s c_i \right) \log \left( \frac{ \sum_{i=1}^{s}  c_i}{\sum_{i=1}^{s}  \frac{1}{\sqrt{i}}} \right) - 2n \geq  2n \log \left( \frac{n}{2\sqrt{s}} \right) - 2n.
\end{align*}  

This quantity is decreasing in $s$.  Using the fact that $s \leq \sqrt{n}$ one obtains 
\[
\Psi(\lambda) \geq 2n\log\!\Big(\frac{n^{3/4}}{2}\Big)-2n=\frac32\,n\log n-2n\log 2-2n\ge \frac32\,n\log n-5n,
\]
as desired, where the final inequality used that $-2\log(2) > -3$. This is clearly not tight, but we need the extra flexibility in the inductive step.  
\end{proof}

 The final observation is the inductive step:  no smallest counterexample can have $s(\lambda) > \sqrt{n}$. Note that Lemma \ref{lem:skinny} serves as a base case for this claim, since it proves the result in particular for $n=1$.  Combined with Lemma \ref{lem:skinny}, this claim completes the proof of Theorem  \ref{thm:Psi}.   

\begin{lemma} 
Suppose $\lambda \vdash n$ and $s = s(\lambda) \geq  \sqrt{n}$.  Then $\lambda$ is not a smallest counterexample to Theorem \ref{thm:Psi}.  \label{lem:strip}
\end{lemma} 
\begin{proof} 

Suppose that $\lambda \vdash n$ with $s \geq \sqrt{n}$ and that all partitions on $n' < n$ satisfy \eqref{eqn:psilb}.  We note that if $\lambda$ consists of a single row, $\Psi(\lambda) = 3 \log n! \geq 3n\log n - 3n$ by Theorem~\ref{thm:crude-stirling}, so \eqref{eqn:psilb} holds. We henceforth assume that $\lambda$ contains at least two rows.   

Let $\mu$ be the partition of $n-s$ obtained by removing the first row of $\lambda$.  Observing that, using Lemma \ref{lem:row_col_prod} regarding the product along the first row of hook lengths, \[
\pi(\lambda) = \pi(\mu) \cdot \prod_{j=1}^{s} h_{\lambda}(1,j) \geq \pi(\mu) \cdot (s!),
\] 	
we obtain
\begin{align}
\Psi(\lambda)
&=2\log(\pi(\lambda))+\sum_{i=1}^{t}\log(\lambda_i!) \nonumber\\
&\ge 2\log(\pi(\mu)) + 2\log(s!)\;+\;\log(s!)\;+\;\sum_{i=2}^{t}\log(\lambda_i!) \nonumber\\
&= \Psi(\mu) + 3\log(s!) \nonumber\\
&\ge \left(\frac32\,(n-s)\log(n-s)-5(n-s)\right) + 3s\log s - 3s, \label{eqn:wts}
\end{align}
It suffices, then, to prove that \eqref{eqn:wts} is at least $\frac{3}{2} n \log n - 5n$.  We first observe that the function
\[
f(x) = \frac{3}{2} (n-x) \log(n-x) + 3x \log x -5n + 2x  
\]
has
\[
f'(x) = -\frac{3}{2} - \frac{3}{2}\log(n-x) + 3 + 3\log(x) + 2 = \frac{3}{2} \log(\frac{x^2}{n-x}) + \frac{7}{2}
\]
which is positive for $x \geq \sqrt{n}$.  This follows as, since $n \geq \sqrt{n}$, one has $\frac{x^2}{n-x} \geq 1$, so the logarithm is positive, and $f'(x) \ge \frac{7}{2}$.
This means, given our condition on $s$, that \eqref{eqn:wts} is minimized when $s = \sqrt{n}$. Thus \begin{align*} 
	\Psi(\lambda)  &\geq \frac{3}{2} (n - \sqrt{n}) \log (n - \sqrt{n}) + 3 \sqrt{n} \log(\sqrt{n}) - 5n + 2\sqrt{n} \\
	&= \frac{3}{2} \left( (n - \sqrt{n}) \log(n) + (n - \sqrt{n}) \log\left(1 -\frac{1}{\sqrt{n}}\right) + 2\sqrt{n} \log(\sqrt{n}) \right) - 5n + 2\sqrt{n} \\
	&= \frac{3}{2} \left( n \log n +  (n - \sqrt{n}) \log\left(1 -\frac{1}{\sqrt{n}}\right) \right) -5n + 2\sqrt{n} \\
	&= \left(\frac{3}{2}n \log n - 5n \right) + \left( \frac{3}{2} (n - \sqrt{n}) \log\left(1 - \frac{1}{\sqrt{n}}\right) + 2\sqrt{n} \right).
\end{align*}    
We can conclude $\Psi(\lambda)>\frac{3}{2}n\log -5n$, as desired, once we show that \[
 \frac{3}{2} (n - \sqrt{n}) \log\left(1 - \frac{1}{\sqrt{n}}\right) + 2\sqrt{n} \geq 0.
\] 

Indeed, for $n \geq 4$ we can apply the fact that $\log(1-x)\ge -x-x^2$, which holds at least on the interval $0\le x\le 1/2$. Plugging in $x=1/\sqrt n$ gives
$\log(1-1/\sqrt n)\ge -(1/\sqrt n)-(1/n)$, and hence
\begin{align*}
\frac{3}{2}(n-\sqrt n)\log\!\Big(1-\frac{1}{\sqrt n}\Big)+2\sqrt n
&\ge \frac{3}{2}(n-\sqrt n)\Big(-\frac{1}{\sqrt n}-\frac{1}{n}\Big)+2\sqrt n\\
&= -\frac{3}{2}\Big(\sqrt n-1\Big)\;-\;\frac{3}{2}\Big(1-\frac{1}{\sqrt n}\Big)\;+\;2\sqrt n\\
&= \frac12\sqrt n+\frac{3}{2\sqrt n}\ >\ 0.
\end{align*}
This proves the inequality for $n\ge 4$. The remaining cases, $n=1,2,3$, can be checked directly. 

\end{proof}

With the above proof, our first main result is an immediate corollary.

\begin{proof}[Proof of Theorem~\ref{thm:main_weak}]
We now combine the lower and upper bounds obtained in the preceding subsections.
The lower bound is given by Proposition~\ref{prop:weak_lower_bound}, and the upper bound is given by Proposition~\ref{prop:Psi-to-prob}. Together these imply the stated asymptotics for $\Prob(\sigma_1\wbo\sigma_2)$ and, equivalently, for $|\weakintervals{n}|$.
\end{proof}

\section{Proof of lower bound on strong Bruhat order}\label{sec:strong}

In this section, we prove Theorem \ref{thm:strong} by giving a construction of many intervals.  Though our construction is deterministic, the analysis is probabilistic and uses properties of simple random walk to show that most of the pairs we construct are comparable.   

 For a fixed $t \geq 1$, we let $\Sigma_1^{(t)}$ denote the set of all permutations which, in one line notation, have the form $\tilde{\sigma}_1 \tilde{\sigma}_2 \tilde{\sigma}_3$, where \begin{itemize} 
\item $\tilde{\sigma}_1$ is a permutation of $\{1,2, \dots, t\}$,
\item $\tilde{\sigma}_2$ is a permutation of $\{t+1, t+2, \dots, n-t\}$,
\item $\tilde{\sigma}_3$ is a permutation of $\{n-t+1, n-t+2, \dots, n\}$.
\end{itemize} 
$\Sigma_2^{(t)}$ is the set of permutations of the form $\tilde{\sigma}_3\tilde{\sigma}_2 \tilde{\sigma}_1$.  Then the main technical tool to derive Theorem \ref{thm:strong} is

\begin{lemma} \label{lem:strongLB}
Suppose $n/2 \geq  t \geq 3 \sqrt{n\log n}.$  Then all but $o(|\Sigma_1^{(t)}||\Sigma_2^{(t)}|)$ of the pairs in $\Sigma_1 ^{(t)} \times \Sigma_2^{(t)}$ are comparable.  
\end{lemma} 

Theorem \ref{thm:strong} then follows immediately from the lemma.

\begin{proof}[Proof of Theorem \ref{thm:strong} from Lemma \ref{lem:strongLB}]
Per Lemma \ref{lem:strongLB}, so long as $n/2\geq t \geq 3 \sqrt{n \log n}$, $|\strongintervals{n}| \geq (1-o(1)) |\Sigma_1^{(t)}||\Sigma_2^{(t)}|.$ 
Let $t = 3 \sqrt{n \log n}$.  Then, using Stirling's approximation,
\begin{align*}
\log(|\Sigma_1^{(t)}|) = \log(|\Sigma_2^{(t)}|) &= \log((t!)^2 (n-2t)!)\\
&=2\log(t!)+\log((n-2t)!)\\
&=2t\log t-2t+(n-2t)\log(n-2t)-(n-2t)+O(\log(n-2t))\\
&=2t\log t+(n-2t)\log n+(n-2t)\log\left(1-\frac{2t}{n}\right)+O(\log (n-2t))\\
&=\log(n!)+2t\log(t/n)+(n-2t)\log\!\left(1-\frac{2t}{n}\right)+O(\log n)\\
&=\log(n!)+2t\log(t/n)-2t+O\!\left(\frac{t^2}{n}+\log n\right),
\end{align*}
where the second to last line uses Stirling's approximation ``backwards'' to introduce a $\log (n!)$ and the last line uses the Taylor expansion $\log(1-x)=-x+O(x^2)$ with $x=2t/n$.
Using $t=3\sqrt{n\log n}$, we have $\frac{t^2}{n}=9\log n=o\!\big(\sqrt n(\log^{3/2}n)\big)$, while
\[
2t\log(t/n)
=6\sqrt{n\log n}\left(\log 3+\tfrac12\log\log n-\tfrac12\log n\right)
=-(3+o(1))\sqrt n\,(\log^{3/2} n).
\]
Therefore,
\[
\log(|\Sigma_1^{(t)}|) = \log(|\Sigma_2^{(t)}|)
=\log(n!)-(3+o(1))\sqrt n\,(\log^{3/2} n).
\]
Exponentiating, we find that 
\[
|\Sigma_1^{(t)}|=|\Sigma_2^{(t)}|=n!\exp\left(-3\sqrt{n}(\log^{3/2} n)+O(\log n)\right).
\]
Thus $|\Sigma_1^{(t)}||\Sigma_2^{(t)}| \geq (n!)^2 \exp(-(6+o(1))\sqrt{n} (\log^{3/2} n))$, yielding the claimed bound of Theorem \ref{thm:strong}.
\end{proof} 

We now turn to proving Lemma \ref{lem:strongLB}.  Of course, not all pairs in $\Sigma_1^{(t)} \times \Sigma_2^{(t)}$ are comparable, but it suffices to prove that a pair uniformly randomly chosen from $\Sigma_1^{(t)} \times \Sigma_2^{(t)}$ is comparable with high probability.

To this end we recall that comparability in the Bruhat order can be cast in terms of comparability in the Gale order. For $\sigma\in\sym{n}$ and $k\in[n]$, denote by $S(\sigma,k)$ the set $\{\sigma(1),\ldots, \sigma(k)\}$. Then, Theorem~\ref{thm:tableau-criterion} says that two permutations $\sigma,\pi\in\sym{n}$ are comparable in the Bruhat order (with $\sigma \beq_B \pi$) if for all $k\in [n]$, 
 \[
S(\sigma, k) \beq_G S(\pi,k).
 \]
 Thus the crux of the proof of Lemma \ref{lem:strongLB} is to verify this fact for almost all pairs $(\sigma, \pi) \in \Sigma_1 \times \Sigma_2.$ 
 
 Note that the structure of a uniformly random permutation in each $\Sigma_i^{(t)}$ consists of three uniformly random permutations (of different sets.) 
 
\begin{proof}[Proof of Lemma \ref{lem:strongLB}]

Take uniformly randomly chosen elements $\pi_1$ and $\pi_2$ of $\Sigma_1^{(t)}$ and $\Sigma_2^{(t)}$ respectively.  We wish to show that, with high probability, $\pi_1 \beq_B \pi_2.$  By definition of $\Sigma_1^{(t)}$, one has $\pi_1 = \sigma_1\sigma_2\sigma_3$ where $\sigma_1, \sigma_2,$ and $\sigma_3$ are uniformly chosen random permutations of $\{1, \dots, t\}$, $\{t+1, \dots, n-t\}$, and $\{n-t+1, \dots, n\}$, respectively.  Likewise, $\pi_2 = \sigma_1'\sigma_2'\sigma_3'$, where $\sigma'_1, \sigma'_2,$ and $\sigma'_3$ are uniformly chosen random permutations of $\{n-t+1, \dots, n\}$, $\{t+1, \dots, n-t\}$, and $\{1, \dots, t\}$, respectively.

We want to verify that, with high probability, each $S(\pi_1,k) \beq_G S(\pi_2,k).$  We do that by considering different ranges of $k$.  

\noindent{\bf Case 1:} $k \leq t$.

In this case, $S(\pi_1,k) \beq_G S(\pi_2,k)$ holds deterministically, as $S(\pi_1,k) \subseteq \{1, \dots, t\}$ while $S(\pi_2,k) \subseteq \{n-t+1, \dots, n\}$ and $t \leq n/2$.

\noindent{\bf Case 2:} $k \geq n-t+1$. 

Again, $S(\pi_1,k) \beq_G S(\pi_2,k)$ holds deterministically.  $S(\pi_1, k)$ contains all of $\{1, \dots, n-t\}$ and some $k-(n-t)$ elements of $\{n-t+1, \dots, n\}$.  Meanwhile $S(\pi_2, k)$ contains all of $\{t+1, \dots, n\}$ and some $k-(n-t)$ elements of $\{1, \dots, t\}.$  That $S(\pi_1,k) \beq_G S(\pi_2,k)$ in this case is easily verified.     

\noindent{\bf Case 3:} $t+1 \leq k \leq n-t$.

This is the most interesting regime, and where the probability comes into place.  For $1 \leq \ell \leq k \leq n$, we must verify that the $\ell$\textsuperscript{th} largest element of $S(\pi_1,k)$ is at most the $\ell$\textsuperscript{th} largest element of $S(\pi_2, k)$.

For a fixed $k$, this process can be modeled as follows: Consider a walk from $(0,0)$ to $(n,0)$ that takes a positive diagonal step from $(i,j)$ to $(i+1,j+1)$ if $n-j+1 \in S(\pi_2,k) \setminus S(\pi_1,k)$, a negative diagonal step from $(i,j)$ to $(i+1,j-1)$ if $n-j+1 \in S(\pi_1,k) \setminus S(\pi_2,k),$ and a horizontal step from $(i,j)$ to $(i+1,j)$ otherwise.  Let $x_m$ denote the height of this walk after $m$ steps.  Then, $x_{m} = |S(\pi_2,k) \cap \{n,n-1, \dots, n-m+1\}| - |S(\pi_1,k) \cap \{n, n-1, \dots, n-m+1\}|$.  Per the definition of the Gale orders, $S(\pi_1,k) \beq_G S(\pi_2,k)$ if and only if this walk never crosses below the $x$-axis.  

For a fixed $t+1 \leq k \leq n-t$ and $k' \leq n$, let $\mathcal{B}_{k,k'}$ denote the event that $x_{k'} < 0$ --- that is, the event that the walk is below the $x$-axis after $k'$ steps. 

Observe that, though $\mathcal{B}_{k,k'}$ makes sense for all $k' \leq n$, it can occur only if $k' \leq n-t$; for, if $k'>n-t$, then $n-k'+1\in S(\pi_1,k)$, so all these steps are horizontal or downwards. Accordingly, the walk cannot be below the $x$-axis at these points.  

Thus, we note that $S(\pi_1,k) \not\beq_G S(\pi_2,k)$ if and only if some $\mathcal{B}_{k,k'}$ occurs for some $1 \leq k' \leq n-t$, and  $\pi_1 \not\beq_B \pi_2$ if and only if $S(\pi_1,k) \not\beq_G S(\pi_2,k)$ for some $k\in[n]$.

Combined, \[
\Prob(\pi_1 \not\beq_B \pi_2)
= \Prob\left( \bigcup_{\substack{t+1 \leq k \leq n-t\\ t+1 \leq k' \leq n}} \mathcal{B}_{k,k'} \right)
\leq \sum_{\substack{t+1 \leq k \leq n-t\\ t+1 \leq k' \leq n}} \Prob(\mathcal{B}_{k,k'}).
\]

We now prove that $\Prob(\mathcal{B}_{k,k'}) = o(n^{-2})$. Fix $t +1 \leq k \leq n-t$, and fix $k' \leq n$.    If $k' \leq t$, the event $\mathcal{B}_{k,k'}$ does not hold by construction, so we may further assume that $k' > t.$

Now, consider the walk defined above.  Note that $\{n, n-1, \dots, n-t+1\} \in S(\pi_2,k) \setminus S(\pi_1, k)$, so the first $t$ steps of the walk are positive diagonal steps.  Now note that each of $S(\pi_i,k) \cap \{t+1, \dots, n-t\}$ is an independent and uniformly random subset of $\{t+1, \dots, n-t\}$ of size $k-t$.  Thus the walk for the next $k'-t$ steps behaves similarly to an unbiased random walk where some of the steps are flat. It is for this reason that the $\Prob(\mathcal{B}_{k,k'})$ are small, as the initial part of the walk forces us to start well above the $x$-axis.  Unfortunately, the events involved are not quite independent because $S(\pi_i,k)$ have size exactly $k$. This means the walk is not quite unbiased, but it will still turn out to be manageable. It is this intuition we now make precise.  

We move from walks derived from uniform random subsets of size $k-t$ to walks derived from sets containing each element with probability $\frac{k-t}{n-2t}$ independently. These have $k-t$ elements in expectation and are easier to study. We will then condition on these random sets having exactly $k-t$ elements, which yields the same distribution of walks that we wish to study.   

To this end, let us now suppose $X^{(k)}_{t+1}, X^{(k)}_{t+2}, \dots, X^{(k)}_{n-t}$, and also $Y^{(k)}_{t+1}, Y^{(k)}_{t+2}, \dots, Y^{(k)}_{n-t}$ are independent Bernoulli random variables, which take the value $1$ with probability $p = \frac{k-t}{n-2t}.$  As observed previously we only need to be concerned when $k' \leq k-t$ and restrict ourselves to this regime.  
Let $S_{k,k'} = \sum_{i=t+1}^{k'} X_{n+1-i}$ and $S'_{k,k'} = \sum_{i=t+1}^{k'} Y_{n+1-i}$.  Then $W_{k,k'} = t+(S_{k,k'} - S'_{k,k'})$ can be thought of as the final height of a random walk, where the walk takes a positive vertical step at $\{i:X_i^{(k)}-Y_i^{(k)}=1\}$, a negative vertical step at $\{i:Y_i^{(k)}-X_i^{(k)}=1\}$ and a horizontal step otherwise. In other words, we view $X_i^{(k)}$ and $Y_i^{(k)}$ as indicator vectors for our sets $S(\pi_i,k)\cap\{t+1,\ldots, n-t\}$. However, since we need our walks to be derived from sets $S(\pi_i,k)$ satisfying $|S(\pi_i,k)\cap\{t+1,\ldots, n-t\}|=k-t$, whereas there is no bound on the number of $i$ such that $X_i^{(k)}=1$ or $Y_i^{(k)}=1$, we study $W_{k,k'}$ conditioned on $\sum_{i=t+1}^{n-t}X_{n+1-i}^{(k)}=\sum_{i=t+1}^{n-t}Y_{n+1-i}^{(k)}=k-t$.

Let $\mathcal{B}'_{k,k'}$ be the event that $W_{k,k'} < 0$, and let $\mathcal{I}$ denote the event that $\sum_{i=t+1}^{n-t}X_{n+1-i}^{(k)}=\sum_{i=t+1}^{n-t}Y_{n+1-i}^{(k)} = k-t$.  Note that conditioned on $\mathcal{I}$, the sets indexed by the $X_i^{(k)}$ and $Y_i^{(k)}$ have size $k-t$ and are equally likely to be any such set.  Moreover, $\Prob(\mathcal{I}) > \frac{1}{n^2}$.  This follows as $k-t$ is the modal value of $\sum X_i^{(k)}$ and $\sum Y_i^{(k)}$ and these binomial random variables take on only $n-2t < n$ values. We remark that one can get slightly sharper bounds by computing, using Stirling's approximation, $\Prob(\operatorname{Bin}(n-2t,p)) = k-t,$ but this crude bound suffices for our purposes.

Thus\[
\Prob(\mathcal{B}_{k,k'}) = \Prob(\mathcal{B}'_{k,k'}|\mathcal{I}) = \frac{\Prob(\mathcal{B}'_{k,k'} \cap \mathcal{I})}{\Prob(\mathcal{I})} \leq n^2 \Prob(\mathcal{B'}_{k,k'}).  
\]
It suffices then to estimate $\Prob(\mathcal{B'}_{k,k'}).$  Note that $W_{k,k'}$ is a function of the $X_j^{(k)}$ and $Y_j^{(k)}$, and the value changes by at most one if one of the $X_{j}^{(k)}$ or $Y_{j}^{(k)}$ are changed; thus we may apply Proposition~\ref{prop:mcdiarmid} where each $c_j = 1$, and there are $2(k'-t) < 2n$ variables involved.  Moreover, $\e[W_{k,k'}] = t$ as the walk is unbiased after the first $t$ biased steps.  Then
\begin{align*} 
\Prob(\mathcal{B'}_{k,k'}) = \Prob(W_{k,k'} < 0) &= \Prob\big(W_{k,k'} < \e[W_{k,k'}] - t\big) \\&\leq \exp\left( \frac{-2t^2}{2n}  \right) \leq \exp(-9 \log n) = n^{-9}.  
\end{align*} 
Then $\Prob(\mathcal{B}_{k,k'}) \leq n^2 \Prob({B'}_{k,k'}) = o(n^{-2}).$ We have already observed this suffices to prove $\mathcal{P}(\pi_{1} \not\beq_{B} \pi_2) = o(1).$  That is, with probability $1-o(1)$, a pair of randomly selected permutations from $\Sigma_1^{(t)} \times \Sigma_2^{(t)}$ are comparable, and this is what the lemma asserts.  
\end{proof}

\printbibliography

@article {BP21,
    AUTHOR = {Bochkov, I. A. and Petrov, F. V.},
     TITLE = {The bounds for the number of linear extensions via chain and
              antichain coverings},
   JOURNAL = {Order},
  FJOURNAL = {Order. A Journal on the Theory of Ordered Sets and its
              Applications},
    VOLUME = {38},
      YEAR = {2021},
    NUMBER = {2},
     PAGES = {323--328},
      ISSN = {0167-8094,1572-9273},
   MRCLASS = {06A07},
  MRNUMBER = {4274656},
MRREVIEWER = {Andr\'{e}\ Guerino\ Castoldi},
       DOI = {10.1007/s11083-020-09542-3},
       URL = {https://doi.org/10.1007/s11083-020-09542-3},
}

@book {S99,
    AUTHOR = {Stanley, Richard P.},
     TITLE = {Enumerative combinatorics. {V}ol. 2},
    SERIES = {Cambridge Studies in Advanced Mathematics},
    VOLUME = {62},
      NOTE = {With a foreword by Gian-Carlo Rota and appendix 1 by Sergey
              Fomin},
 PUBLISHER = {Cambridge University Press, Cambridge},
      YEAR = {1999},
     PAGES = {xii+581},
   MRCLASS = {05A15 (05-02 05E05 05E10 68R05)},
  MRNUMBER = {1676282},
MRREVIEWER = {Ira\ Gessel},
       DOI = {10.1017/CBO9780511609589},
       URL = {https://doi.org/10.1017/CBO9780511609589},
}

@book {BB05,
    AUTHOR = {Bj\"{o}rner, Anders and Brenti, Francesco},
     TITLE = {Combinatorics of {C}oxeter groups},
    SERIES = {Graduate Texts in Mathematics},
    VOLUME = {231},
 PUBLISHER = {Springer, New York},
      YEAR = {2005},
     PAGES = {xiv+363},
   MRCLASS = {05-01 (05E15 20F55)},
  MRNUMBER = {2133266},
MRREVIEWER = {Jian-yi\ Shi},
}

@article {M71,
    AUTHOR = {Mirsky, L.},
     TITLE = {A dual of {D}ilworth's decomposition theorem},
   JOURNAL = {Amer. Math. Monthly},
  FJOURNAL = {American Mathematical Monthly},
    VOLUME = {78},
      YEAR = {1971},
     PAGES = {876--877},
      ISSN = {0002-9890,1930-0972},
   MRCLASS = {06.20},
  MRNUMBER = {288054},
MRREVIEWER = {Richard\ Stanley},
       DOI = {10.2307/2316481},
       URL = {https://doi.org/10.2307/2316481},
}

@article {G74,
    AUTHOR = {Greene, Curtis},
     TITLE = {An extension of {S}chensted's theorem},
   JOURNAL = {Advances in Math.},
  FJOURNAL = {Advances in Mathematics},
    VOLUME = {14},
      YEAR = {1974},
     PAGES = {254--265},
      ISSN = {0001-8708},
   MRCLASS = {05A17},
  MRNUMBER = {354395},
MRREVIEWER = {A.\ O.\ Morris},
       DOI = {10.1016/0001-8708(74)90031-0},
       URL = {https://doi.org/10.1016/0001-8708(74)90031-0},
}

@article {BW91,
    AUTHOR = {Bj\"{o}rner, Anders and Wachs, Michelle L.},
     TITLE = {Permutation statistics and linear extensions of posets},
   JOURNAL = {J. Combin. Theory Ser. A},
  FJOURNAL = {Journal of Combinatorial Theory. Series A},
    VOLUME = {58},
      YEAR = {1991},
    NUMBER = {1},
     PAGES = {85--114},
      ISSN = {0097-3165,1096-0899},
   MRCLASS = {06A07 (05E99 20B35)},
  MRNUMBER = {1119703},
MRREVIEWER = {Sergey\ V.\ Fomin},
       DOI = {10.1016/0097-3165(91)90075-R},
       URL = {https://doi.org/10.1016/0097-3165(91)90075-R},
}

@article {BDJ99,
    AUTHOR = {Baik, Jinho and Deift, Percy and Johansson, Kurt},
     TITLE = {On the distribution of the length of the longest increasing
              subsequence of random permutations},
   JOURNAL = {J. Amer. Math. Soc.},
  FJOURNAL = {Journal of the American Mathematical Society},
    VOLUME = {12},
      YEAR = {1999},
    NUMBER = {4},
     PAGES = {1119--1178},
      ISSN = {0894-0347,1088-6834},
   MRCLASS = {05A05 (33D45 45E05 60C05)},
  MRNUMBER = {1682248},
MRREVIEWER = {David\ J.\ Aldous},
       DOI = {10.1090/S0894-0347-99-00307-0},
       URL = {https://doi.org/10.1090/S0894-0347-99-00307-0},
}

@inproceedings {HR18,
    AUTHOR = {Hardy, G. H. and Ramanujan, S.},
     TITLE = {Asymptotic formul\ae in combinatory analysis [{P}roc. {L}ondon
              {M}ath. {S}oc. (2) {\bf 17} (1918), 75--115]},
 BOOKTITLE = {Collected papers of {S}rinivasa {R}amanujan},
     PAGES = {276--309},
 PUBLISHER = {AMS Chelsea Publ., Providence, RI},
      YEAR = {2000},
      ISBN = {0-8218-2076-1},
   MRCLASS = {01A75},
  MRNUMBER = {2280879},
}

@article {R55,
    AUTHOR = {Robbins, Herbert},
     TITLE = {A remark on {S}tirling's formula},
   JOURNAL = {Amer. Math. Monthly},
  FJOURNAL = {American Mathematical Monthly},
    VOLUME = {62},
      YEAR = {1955},
     PAGES = {26--29},
      ISSN = {0002-9890,1930-0972},
   MRCLASS = {33.0X},
  MRNUMBER = {69328},
MRREVIEWER = {S.\ C.\ van Veen},
       DOI = {10.2307/2308012},
       URL = {https://doi.org/10.2307/2308012},
}

@article {HP08,
    AUTHOR = {Hammett, Adam and Pittel, Boris},
     TITLE = {How often are two permutations comparable?},
   JOURNAL = {Trans. Amer. Math. Soc.},
  FJOURNAL = {Transactions of the American Mathematical Society},
    VOLUME = {360},
      YEAR = {2008},
    NUMBER = {9},
     PAGES = {4541--4568},
      ISSN = {0002-9947,1088-6850},
   MRCLASS = {05A05},
  MRNUMBER = {2403696},
MRREVIEWER = {Mikl\'{o}s\ B\'{o}na},
       DOI = {10.1090/S0002-9947-08-04478-4},
       URL = {https://doi.org/10.1090/S0002-9947-08-04478-4},
}

@article {Ehresmann34,
    AUTHOR = {Ehresmann, Charles},
     TITLE = {Sur la topologie de certains espaces homog\`enes},
   JOURNAL = {Ann. of Math. (2)},
  FJOURNAL = {Annals of Mathematics. Second Series},
    VOLUME = {35},
      YEAR = {1934},
    NUMBER = {2},
     PAGES = {396--443},
      ISSN = {0003-486X,1939-8980},
   MRCLASS = {DML},
  MRNUMBER = {1503170},
       DOI = {10.2307/1968440},
       URL = {https://doi.org/10.2307/1968440},
}

@article {Verma71,
    AUTHOR = {Verma, Daya-Nand},
     TITLE = {M\"{o}bius inversion for the {B}ruhat ordering on a {W}eyl
              group},
   JOURNAL = {Ann. Sci. \'{E}cole Norm. Sup. (4)},
  FJOURNAL = {Annales Scientifiques de l'\'{E}cole Normale Sup\'{e}rieure.
              Quatri\`eme S\'{e}rie},
    VOLUME = {4},
      YEAR = {1971},
     PAGES = {393--398},
      ISSN = {0012-9593},
   MRCLASS = {06A55 (20G20)},
  MRNUMBER = {291045},
MRREVIEWER = {P.\ F.\ Conrad},
       URL = {http://www.numdam.org/item?id=ASENS_1971_4_4_3_393_0},
}

@article {BjornerWachs82,
    AUTHOR = {Bj\"{o}rner, Anders and Wachs, Michelle},
     TITLE = {Bruhat order of {C}oxeter groups and shellability},
   JOURNAL = {Adv. in Math.},
  FJOURNAL = {Advances in Mathematics},
    VOLUME = {43},
      YEAR = {1982},
    NUMBER = {1},
     PAGES = {87--100},
      ISSN = {0001-8708},
   MRCLASS = {20H15 (06A10 13F20 14L30 52A43)},
  MRNUMBER = {644668},
MRREVIEWER = {S.\ Milne},
       DOI = {10.1016/0001-8708(82)90029-9},
       URL = {https://doi.org/10.1016/0001-8708(82)90029-9},
}

@article {KazhdanLusztig79,
    AUTHOR = {Kazhdan, David and Lusztig, George},
     TITLE = {Representations of {C}oxeter groups and {H}ecke algebras},
   JOURNAL = {Invent. Math.},
  FJOURNAL = {Inventiones Mathematicae},
    VOLUME = {53},
      YEAR = {1979},
    NUMBER = {2},
     PAGES = {165--184},
      ISSN = {0020-9910,1432-1297},
   MRCLASS = {20H15 (17B35 20G05 22E47)},
  MRNUMBER = {560412},
MRREVIEWER = {Vinay\ V.\ Deodhar},
       DOI = {10.1007/BF01390031},
       URL = {https://doi.org/10.1007/BF01390031},
}

@incollection {Bjorner84Orderings,
    AUTHOR = {Bj\"{o}rner, Anders},
     TITLE = {Orderings of {C}oxeter groups},
 BOOKTITLE = {Combinatorics and algebra ({B}oulder, {C}olo., 1983)},
    SERIES = {Contemp. Math.},
    VOLUME = {34},
     PAGES = {175--195},
 PUBLISHER = {Amer. Math. Soc., Providence, RI},
      YEAR = {1984},
      ISBN = {0-8218-5029-6},
   MRCLASS = {05A99 (06F15 14M15 20H15 52A25)},
  MRNUMBER = {777701},
MRREVIEWER = {Robert\ A.\ Proctor},
       DOI = {10.1090/conm/034/777701},
       URL = {https://doi.org/10.1090/conm/034/777701},
}

@book {Humphreys90,
    AUTHOR = {Humphreys, James E.},
     TITLE = {Reflection groups and {C}oxeter groups},
    SERIES = {Cambridge Studies in Advanced Mathematics},
    VOLUME = {29},
 PUBLISHER = {Cambridge University Press, Cambridge},
      YEAR = {1990},
     PAGES = {xii+204},
      ISBN = {0-521-37510-X},
   MRCLASS = {20-02 (20F32 20F55 20G15 20H15)},
  MRNUMBER = {1066460},
MRREVIEWER = {Louis\ Solomon},
       DOI = {10.1017/CBO9780511623646},
       URL = {https://doi.org/10.1017/CBO9780511623646},
}

@book {Ziegler95,
    AUTHOR = {Ziegler, G\"{u}nter M.},
     TITLE = {Lectures on polytopes},
    SERIES = {Graduate Texts in Mathematics},
    VOLUME = {152},
 PUBLISHER = {Springer-Verlag, New York},
      YEAR = {1995},
     PAGES = {x+370},
      ISBN = {0-387-94365-X},
   MRCLASS = {52Bxx},
  MRNUMBER = {1311028},
MRREVIEWER = {Margaret\ M.\ Bayer},
       DOI = {10.1007/978-1-4613-8431-1},
       URL = {https://doi.org/10.1007/978-1-4613-8431-1},
}

@incollection{C94,
  author    = {Chevalley, Claude},
  title     = {Sur les d{\'e}compositions cellulaires des espaces {$G/B$}},
  booktitle = {Algebraic Groups and Their Generalizations: Classical Methods},
  series    = {Proc. Sympos. Pure Math.},
  volume    = {56},
  number    = {Part 1},
  pages     = {1--23},
  publisher = {Amer. Math. Soc.},
  address   = {Providence, RI},
  year      = {1994},
  note      = {With a foreword by Armand Borel; proceedings of University Park, PA, 1991},
}

@book {KnuthTAOCP3,
    AUTHOR = {Knuth, Donald E.},
     TITLE = {The art of computer programming. {V}ol. 3},
      NOTE = {Sorting and searching,
              Second edition [of MR0445948]},
 PUBLISHER = {Addison-Wesley, Reading, MA},
      YEAR = {1998},
     PAGES = {xiv+780},
      ISBN = {0-201-89685-0},
   MRCLASS = {68-02},
  MRNUMBER = {3077154},
}

@article {H02,
    AUTHOR = {Hilbert, David},
     TITLE = {Mathematical problems},
   JOURNAL = {Bull. Amer. Math. Soc.},
  FJOURNAL = {Bulletin of the American Mathematical Society},
    VOLUME = {8},
      YEAR = {1902},
    NUMBER = {10},
     PAGES = {437--479},
      ISSN = {0002-9904},
   MRCLASS = {DML},
  MRNUMBER = {1557926},
       DOI = {10.1090/S0002-9904-1902-00923-3},
       URL = {https://doi.org/10.1090/S0002-9904-1902-00923-3},
}

@article{GR63,
  author = {Guilbaud, G. Th. and Rosenstiehl, P.},
  title = {Analyse alg\'ebrique d'un scrutin},
  journal = {Math\'ematiques et sciences humaines},
  pages = {9--33},
  year = {1963},
  publisher = {Ecole Pratique des hautes \'etudes, Centre de math\'ematique sociale et de statistique},
  volume = {4},
  language = {fr},
  url = {https://www.numdam.org/item/MSH_1963__4__9_0/}
}

@article {Sch61,
    AUTHOR = {Schensted, C.},
     TITLE = {Longest increasing and decreasing subsequences},
   JOURNAL = {Canadian J. Math.},
  FJOURNAL = {Canadian Journal of Mathematics. Journal Canadien de
              Math\'ematiques},
    VOLUME = {13},
      YEAR = {1961},
     PAGES = {179--191},
      ISSN = {0008-414X,1496-4279},
   MRCLASS = {05.00},
  MRNUMBER = {121305},
MRREVIEWER = {D.\ E.\ Rutherford},
       DOI = {10.4153/CJM-1961-015-3},
       URL = {https://doi.org/10.4153/CJM-1961-015-3},
}

@article {Chapoton_tamari05,
    AUTHOR = {Chapoton, F.},
     TITLE = {Sur le nombre d'intervalles dans les treillis de {T}amari},
   JOURNAL = {S\'em. Lothar. Combin.},
  FJOURNAL = {S\'eminaire Lotharingien de Combinatoire},
    VOLUME = {55},
      YEAR = {2005},
     PAGES = {Art. B55f, 18},
      ISSN = {1286-4889},
   MRCLASS = {05A15 (05A19 05C05)},
  MRNUMBER = {2264942},
MRREVIEWER = {Joseph\ Kung},
}

@article {MFPR_mTamari11,
    AUTHOR = {Bousquet-M\'elou, Mireille and Fusy, \'Eric and
              Pr\'eville-Ratelle, Louis-Fran\c cois},
     TITLE = {The number of intervals in the {$m$}-{T}amari lattices},
   JOURNAL = {Electron. J. Combin.},
  FJOURNAL = {Electronic Journal of Combinatorics},
    VOLUME = {18},
      YEAR = {2011},
    NUMBER = {2},
     PAGES = {Paper 31, 26},
      ISSN = {1077-8926},
   MRCLASS = {05A15 (06A07)},
  MRNUMBER = {2880681},
MRREVIEWER = {Aaron\ J.\ Robertson},
       DOI = {10.37236/2027},
       URL = {https://doi.org/10.37236/2027},
}

@incollection {SCV_stanley86,
    AUTHOR = {de Sainte-Catherine, Myriam and Viennot, G\'erard},
     TITLE = {Enumeration of certain {Y}oung tableaux with bounded height},
 BOOKTITLE = {Combinatoire \'enum\'erative ({M}ontreal, {Q}ue.,
              1985/{Q}uebec, {Q}ue., 1985)},
    SERIES = {Lecture Notes in Math.},
    VOLUME = {1234},
     PAGES = {58--67},
 PUBLISHER = {Springer, Berlin},
      YEAR = {1986},
      ISBN = {3-540-17207-6},
   MRCLASS = {05A15 (05A10 15A15)},
  MRNUMBER = {927758},
MRREVIEWER = {Joseph\ Kung},
       DOI = {10.1007/BFb0072509},
       URL = {https://doi.org/10.1007/BFb0072509},
}

@article {Kreweras72,
    AUTHOR = {Kreweras, G.},
     TITLE = {Sur les partitions non crois\'ees d'un cycle},
   JOURNAL = {Discrete Math.},
  FJOURNAL = {Discrete Mathematics},
    VOLUME = {1},
      YEAR = {1972},
    NUMBER = {4},
     PAGES = {333--350},
      ISSN = {0012-365X,1872-681X},
   MRCLASS = {05A17 (06A20)},
  MRNUMBER = {309747},
MRREVIEWER = {Robin\ J.\ Wilson},
       DOI = {10.1016/0012-365X(72)90041-6},
       URL = {https://doi.org/10.1016/0012-365X(72)90041-6},
}

@article {Pittel99,
    AUTHOR = {Pittel, Boris},
     TITLE = {Confirming two conjectures about the integer partitions},
   JOURNAL = {J. Combin. Theory Ser. A},
  FJOURNAL = {Journal of Combinatorial Theory. Series A},
    VOLUME = {88},
      YEAR = {1999},
    NUMBER = {1},
     PAGES = {123--135},
      ISSN = {0097-3165,1096-0899},
   MRCLASS = {11P82 (05A17 05C07 60C05)},
  MRNUMBER = {1713480},
MRREVIEWER = {Michael\ Szalay},
       DOI = {10.1006/jcta.1999.2986},
       URL = {https://doi.org/10.1006/jcta.1999.2986},
}

@article {Pittel97,
    AUTHOR = {Pittel, Boris},
     TITLE = {Random set partitions: asymptotics of subset counts},
   JOURNAL = {J. Combin. Theory Ser. A},
  FJOURNAL = {Journal of Combinatorial Theory. Series A},
    VOLUME = {79},
      YEAR = {1997},
    NUMBER = {2},
     PAGES = {326--359},
      ISSN = {0097-3165,1096-0899},
   MRCLASS = {05A18},
  MRNUMBER = {1462562},
MRREVIEWER = {L.\ Bruce\ Richmond},
       DOI = {10.1006/jcta.1997.2791},
       URL = {https://doi.org/10.1006/jcta.1997.2791},
}

@book{GKP94,
  author    = {Graham, Ronald L. and Knuth, Donald E. and Patashnik, Oren},
  title     = {Concrete Mathematics: A Foundation for Computer Science},
  edition   = {2},
  publisher = {Addison-Wesley},
  address   = {Reading, MA},
  year      = {1994},
  isbn      = {978-0-201-55802-9},
  note      = {ISBN-10: 0201558025. xiii+657 pp.},
  doi       = {10.5555/562056},
}

@article {Jensen06,
    AUTHOR = {Jensen, J. L. W. V.},
     TITLE = {Sur les fonctions convexes et les in\'{e}galit\'{e}s entre les
              valeurs moyennes},
   JOURNAL = {Acta Math.},
  FJOURNAL = {Acta Mathematica},
    VOLUME = {30},
      YEAR = {1906},
    NUMBER = {1},
     PAGES = {175--193},
      ISSN = {0001-5962,1871-2509},
   MRCLASS = {DML},
  MRNUMBER = {1555027},
       DOI = {10.1007/BF02418571},
       URL = {https://doi.org/10.1007/BF02418571},
}

@book{Stirling1730,
  author    = {Stirling, James},
  title     = {Methodus Differentialis, sive Tractatus de Summatione et Interpolatione Serierum Infinitarum},
  address   = {London},
  publisher = {Typis Gul.\ Bowyer; impensis G.\ Strahan},
  year      = {1730}
}

@incollection {M89,
    AUTHOR = {McDiarmid, Colin},
     TITLE = {On the method of bounded differences},
 BOOKTITLE = {Surveys in combinatorics, 1989 ({N}orwich, 1989)},
    SERIES = {London Math. Soc. Lecture Note Ser.},
    VOLUME = {141},
     PAGES = {148--188},
 PUBLISHER = {Cambridge Univ. Press, Cambridge},
      YEAR = {1989},
      ISBN = {0-521-37823-0},
   MRCLASS = {05C80 (60E15 60F10 60G42)},
  MRNUMBER = {1036755},
MRREVIEWER = {Alan\ M.\ Frieze},
}

@article {Greene76,
    AUTHOR = {Greene, Curtis},
     TITLE = {Some partitions associated with a partially ordered set},
   JOURNAL = {J. Combinatorial Theory Ser. A},
  FJOURNAL = {Journal of Combinatorial Theory. Series A},
    VOLUME = {20},
      YEAR = {1976},
    NUMBER = {1},
     PAGES = {69--79},
      ISSN = {0097-3165},
   MRCLASS = {06A10 (05A17)},
  MRNUMBER = {398912},
MRREVIEWER = {R.\ P.\ Dilworth},
       DOI = {10.1016/0097-3165(76)90078-9},
       URL = {https://doi.org/10.1016/0097-3165(76)90078-9},
}

@article {Reading04,
    AUTHOR = {Reading, Nathan},
     TITLE = {Lattice congruences of the weak order},
   JOURNAL = {Order},
  FJOURNAL = {Order. A Journal on the Theory of Ordered Sets and its
              Applications},
    VOLUME = {21},
      YEAR = {2004},
    NUMBER = {4},
     PAGES = {315--344 (2005)},
      ISSN = {0167-8094,1572-9273},
   MRCLASS = {20F55 (05E15 06B10 52C35)},
  MRNUMBER = {2209128},
MRREVIEWER = {I.\ Gy.\ Maurer},
       DOI = {10.1007/s11083-005-4803-8},
       URL = {https://doi.org/10.1007/s11083-005-4803-8},
}

@article {Stanley80,
    AUTHOR = {Stanley, Richard P.},
     TITLE = {Weyl groups, the hard {L}efschetz theorem, and the {S}perner
              property},
   JOURNAL = {SIAM J. Algebraic Discrete Methods},
  FJOURNAL = {Society for Industrial and Applied Mathematics. Journal on
              Algebraic and Discrete Methods},
    VOLUME = {1},
      YEAR = {1980},
    NUMBER = {2},
     PAGES = {168--184},
      ISSN = {0196-5212},
   MRCLASS = {20G05 (05A05 06A10 14M17)},
  MRNUMBER = {578321},
       DOI = {10.1137/0601021},
       URL = {https://doi.org/10.1137/0601021},
}

@article {GaetzGao20,
    AUTHOR = {Gaetz, Christian and Gao, Yibo},
     TITLE = {A combinatorial {${sl}_2$}-action and the {S}perner
              property for the weak order},
   JOURNAL = {Proc. Amer. Math. Soc.},
  FJOURNAL = {Proceedings of the American Mathematical Society},
    VOLUME = {148},
      YEAR = {2020},
    NUMBER = {1},
     PAGES = {1--7},
      ISSN = {0002-9939,1088-6826},
   MRCLASS = {06A07 (05E18 06A11)},
  MRNUMBER = {4042823},
MRREVIEWER = {Leila\ Khatami},
       DOI = {10.1090/proc/14655},
       URL = {https://doi.org/10.1090/proc/14655},
}

@book {CT91,
    AUTHOR = {Cover, Thomas M. and Thomas, Joy A.},
     TITLE = {Elements of information theory},
    SERIES = {Wiley Series in Telecommunications},
      NOTE = {A Wiley-Interscience Publication},
 PUBLISHER = {John Wiley \& Sons, Inc., New York},
      YEAR = {1991},
     PAGES = {xxiv+542},
      ISBN = {0-471-06259-6},
   MRCLASS = {94-02 (94A17)},
  MRNUMBER = {1122806},
MRREVIEWER = {L.\ L.\ Campbell},
       DOI = {10.1002/0471200611},
       URL = {https://doi.org/10.1002/0471200611},
}

\end{document}